\documentclass[12pt]{amsart}
\pdfoutput = 1 

\usepackage{amssymb}
\usepackage{amsthm}
\usepackage{amsmath}
\usepackage{mathtools}
\usepackage{amsfonts}
\usepackage{mdframed}
\allowdisplaybreaks
\usepackage{subcaption} 
\usepackage[normalem]{ulem}

\usepackage[skins, breakable]{tcolorbox}

\usepackage[foot]{amsaddr} 

\usepackage[left=1in,right=1in,top=1in,bottom=1in,bindingoffset=0cm]{geometry} 

\usepackage{hyperref}
\hypersetup{
    colorlinks=true,
    linkcolor=blue,
    filecolor=magenta,
    urlcolor=red,
    citecolor=cyan,
}
\usepackage{tikz}
\usetikzlibrary{calc}
\usepackage{enumitem}
\usepackage{standalone} 

\usepackage[backend=biber, firstinits=true, style=numeric, sorting=nyt, maxnames=100,backref=true, sortcites = true]{biblatex}
\DeclareSourcemap{
  \maps[datatype=bibtex, overwrite]{
    \map{
      \step[fieldset=DOI, null]
      \step[fieldset=URL, null]
      \step[fieldset=ISSN, null]
      \step[fieldset=primaryclass, null]
    }
  }
}

\renewbibmacro{in:}{}
\renewbibmacro*{volume+number+eid}{%
	\printfield{volume}%
	\setunit*{\addnbspace}
	\printfield{number}%
	\setunit{\addcomma\space}%
	\printfield{eid}}

\DeclareFieldFormat[article]{volume}{\textbf{#1}\space}
\DeclareFieldFormat[article]{number}{\mkbibparens{#1}}

\DeclareFieldFormat{journaltitle}{#1,}
\DeclareFieldFormat[thesis]{title}{\mkbibemph{#1}\addperiod}
\DeclareFieldFormat[article, unpublished, thesis]{title}{\mkbibemph{#1},}
\DeclareFieldFormat[book]{title}{\mkbibemph{#1}\addperiod}
\DeclareFieldFormat[unpublished]{howpublished}{#1, }

\DeclareFieldFormat{pages}{#1}

\DeclareFieldFormat[article]{series}{Ser.~#1\addcomma}

\newtheorem{theorem}{Theorem}[section] 
\newtheorem{lemma}[theorem]{Lemma}
\newtheorem{corollary}[theorem]{Corollary}
 
\newtheorem{question}[theorem]{Question}

\newtheorem*{theorem*}{Theorem}

\theoremstyle{definition} 
\newtheorem{definition}[theorem]{Definition}

\theoremstyle{remark} 

\newtheorem{remark}[theorem]{Remark}

\newcommand*{\myproofname}{Proof}

\newcommand{\N}{{\mathbb{N}}}

\renewcommand{\epsilon}{\varepsilon}

\renewcommand{\phi}{\varphi}
\renewcommand{\leq}{\leqslant}
\renewcommand{\geq}{\geqslant}
\newcommand{\defeq}{\coloneqq}
\newcommand{\im}{\mathrm{im}}
\newcommand{\dom}{\mathrm{dom}}

\newcommand{\emphdef}[1]{\textit{{#1}}}
\renewcommand{\mod}{\,\,\text{mod}\,\,}
\renewcommand{\int}{\text{int}}

\newcommand{\notsim}{\not\sim}
\NewCommandCopy{\Hungarian}{\H}
\renewcommand{\H}{\mathcal{H}}

\newcommand{\defect}{\text{def}}

\usepackage{titlesec}
\titleformat{\section}[block]{\large\scshape\center}{\thesection.}{1ex}{}
\titleformat{\subsection}[block]{\bfseries}{\thesubsection.}{1ex}{}
\titleformat{\subsubsection}[runin]{\itshape}{\bfseries\upshape\thesubsubsection.}{1ex}{}[.---]

\titlespacing*{\section}{0pt}{*3}{*1}
\titlespacing*{\subsection}{0pt}{*3}{*1}
\titlespacing*{\subsubsection}{0pt}{*1.5}{*0}

\usepackage{caption}

\begin{document}
\emergencystretch 1em 

\title[Defective Correspondence Coloring of planar graphs]{Defective Correspondence Coloring of planar graphs}

\author{James Anderson}
\address{School of Mathematics, Georgia Institute of Technology, Atlanta GA, 30318}
\email{james.anderson@math.gatech.edu}


\begin{abstract}
Defective coloring (also  known as relaxed or improper coloring) is a generalization of proper coloring defined as follows: for $d \in \mathbb{N}$, a coloring of a graph is \textit{$d$-defective} if every vertex is colored the same as at most $d$ of its neighbors. We investigate defective coloring of planar graphs in the context of \textit{correspondence coloring}, a generalization of list coloring introduced by Dvo{\v{r}}{\'a}k and Postle. First we show there exists a planar graph that is not $3$-defective $3$-correspondable, strengthening a recent result of Cho, Choi, Kim, Park, Shan, and Zhu. Then we construct a planar graph that is $1$-defective $3$-correspondable but not $4$-correspondable, thereby extending a recent result of Ma, Xu, and Zhu from list coloring to correspondence coloring. Finally we show all outerplanar graphs are $3$-defective $2$-correspondence colorable, with 3 defects being best possible.
\end{abstract}

\maketitle

\section{Introduction}
All graphs considered in this paper are finite, undirected, and simple. For a natural number $n$, we let $[n]$ denote $\{1, 2, \ldots, n\}$. For a function $f$, we let $\dom(f)$ denote the domain of $f$ and $\im(f)$ denote the image of $f$.

\textit{Defective} coloring (also known as \textit{relaxed} or \textit{improper} coloring) of a graph $G$ is defined as follows: for $d \in \N$, a vertex coloring of a graph $G$ is \textit{$d$-defective} if every vertex is colored the same as at most $d$ of its neighbors, i.e., each color class of vertices induces a subgraph of $G$ of degree at most $d$.
Formally, given a graph $G$ and a coloring $\phi : V(G) \to \N$, the \emphdef{defect} of a vertex $v \in V(G)$ under $\phi$, denoted by $\defect_\phi(v)$ or simply $\defect(v)$ when $\phi$ is understood, is the number of neighbors $u \in N(v)$ that receive the same color as $v$, i.e., satisfy $\phi(u) = \phi(v)$. For $d, k \in \N$, a $k$-coloring $\phi$ of $G$ is \textit{$d$-defective ($d$-def)} if $\defect(v) \leq d$ for all $v \in V(G)$. 
A proper coloring is a $0$-def coloring, and so defective coloring can be viewed as a natural generalization of proper coloring.

In this paper, we investigate defective coloring of planar graphs in the context of \textit{correspondence coloring}, a generalization of list coloring introduced by Dvo{\v{r}}{\'a}k and Postle
\cite{dvovrak2018correspondence} (we refer the reader unfamiliar with list and correspondence coloring to \S\ref{sec: correspondence}, where we introduce the necessary definitions and terminology). Defective coloring can naturally be extended to correspondence coloring (and thus list coloring) as follows:

\begin{definition} Let $G$ be a graph and $\H = (L, H)$ a correspondence cover for $G$. Then an $\H$-coloring $\phi$ of $G$ is \emphdef{$d$-defective} if the image of $\phi$ induces a subgraph of degree at most $d$ in $H$ (i.e., the color $\phi(v)$ for $v \in V(G)$ conflicts with $\phi(u)$ for at most $d$ neighbors $u \in N(v)$). A graph $G$ is \emphdef{$d$-def $k$-correspondable} if there exists a $d$-def $\H$-coloring of $G$ for every $k$-fold correspondence cover $\H$ for $G$. 
\end{definition}
Since list coloring is a special case of correspondence coloring, the definition naturally extends to the list coloring case: when there exists a $d$-def $L$-coloring of $G$ for every $k$-fold list assignment $L$ for $G$, we say $G$ is \textit{$d$-def $k$-choosable}. We note that when a graph is $0$-def $k$-choosable or $0$-def $k$-correspondable we will simply say it is \textit{$k$-choosable} or \textit{$k$-correspondable}, respectively.

\subsection{History of Defective Coloring}\label{subsection: history of defective coloring}
Although the idea of defective vertex coloring has appeared in a handful of earlier works under different names (e.g., decomposition theorems such as that of Lov\'asz \cite{lovasz1966decomposition}), Cowen--Cowen-- Woodall formally introduced defective coloring in \cite{cowen1986defective}, where they focused on planar graphs. They proved every planar graph is $2$-def $3$-colorable and that this cannot be improved to $1$-def $3$-colorable. In addition, the authors demonstrate that there does not exist $d \in \N$ such that all planar graphs are $d$-def $2$-colorable. As the 4-color Theorem \cite{robertson1997four} is equivalent to the statement that all planar graphs are $0$-def $4$-colorable, these results completely characterize the pairs $(d,k)$ for which all planar graphs are $d$-def $k$-colorable.

Since then, defective colorings have received much attention, with hundreds of papers written on the subject and its variations (see the dynamic survey of Wood \cite{wood2018defective}).
One such generalization of defective coloring is defective \textit{list coloring}, introduced independently by \v{S}krekovski \cite{vskrekovski1999list} and Eaton--Hull \cite{eaton1999defective}.
In their works,
\v{S}krekovski \cite{vskrekovski1999list} and Eaton--Hull \cite{eaton1999defective} used Thomassen-style arguments to generalize the result of Cowen--Cowen--Woodall, proving that all planar graphs are $2$-def $3$-choosable. As there are planar graphs that are \textit{not} $4$-choosable \cite{voigt1993list}, the authors then asked if all planar graphs are $1$-def $4$-choosable. This question went unresolved for nearly 15 years until Cushing--Kierstead  \cite{cushing2010planar} answered it in the affirmative, again using Thomassen-style techniques.
As Thomassen proved every planar graph is $5$-choosable \cite{thomassen1994every}, the result of Cushing--Kierstead completely resolved the question of when a planar graph is $d$-def $k$-choosable for all pairs $(d,k)$. 

Another line of inquiry is the relationship \textit{between} proper and defective (list) coloring.
For instance, if a planar graph is $4$-choosable, then is it $1$-def $3$-choosable? The answer to this question is negative---indeed, Cowen--Cowen--Woodall \cite{cowen1986defective} constructed a planar graph that is $4$-choosable (it is 3-degenerate) but that is not $1$-def $3$-colorable (and thus is not $1$-def $3$-choosable).

The converse was asked by Wang and Xu in \cite[Problem 1]{wang2013improper}: is every $1$-def $3$-choosable planar graph $4$-choosable? A recent paper by Ma--Xu--Zhu \cite{ma2023two} answers this question in the negative, constructing a planar graph that is $1$-def $3$-choosable, but not $4$-choosable. This completely resolved the question of the relationship between proper and defective list coloring for planar graphs.

Thus the natural next step is to consider these results in the context of defective correspondence coloring.

\subsection{History of Defective Correspondence Coloring}
Defective correspondence coloring was introduced recently in a series of papers by Kostoshka--Xu \cite{kostochka20212}, Jing--Kostochka--Ma--Xu \cite{jing2022defective}, and Nakprasit--Sittitrai \cite{sittitrai2019sufficient}. The first two papers focused on coloring sparse graphs (not necessarily planar) using 2 colors. Among other results, they determine the minimum number of edges of graphs that are $2$-def $2$-correspondence critical (i.e., graphs that are not $2$-def $2$-correspondable but with the property that every proper subgraph is). They also found sufficient conditions on the maximum average degree of a graph that guarantees it is $2$-def $2$-correspondable, improving known results for list coloring.

In \cite{sittitrai2019sufficient}, Nakprasit--Sittitrai prove every planar graph without cycles of length 4 and 6 is $1$-def $3$-correspondable, generalizing the list-version of the result given by Lih--Song--Wang--Zhang \cite{lih2001note}.
Since these papers, several other articles on defective correspondence coloring have appeared in the literature, including \cite{cai2021note, sittitrai2019sufficient, sribunhung2023relaxed, xiao2024weak, fang2022relaxed, sittitrai2023weak, kostochka2024sparse, jing2021defective}. These results have either focused on coloring graphs (not necessarily planar) restricted to 2 colors, or have focused on planar graphs without cycles of various lengths.

A natural question is therefore the following:

\begin{question}
    For which pairs $(d,k) \in \N\times\N$ is every planar graph $d$-def $k$-correspondable? 
\end{question}

As the proof of Thomassen's result that all planar graphs are $5$-choosable \cite{thomassen1994every} can easily be adapted to the correspondence setting \cite{dvovrak2018correspondence}, it follows all planar graphs are $0$-def $5$-correspondable. As mentioned in \S\ref{subsection: history of defective coloring}, there is no $d$ such that all planar graphs are $d$-def $2$-colorable, and so we turn our attention to $k \in \{3, 4\}$.

It follows from recent results of Cho--Choi--Kim--Park--Shan--Zhu \cite{cho2022decomposing} on edge decompositions that every planar graph is $2$-def $4$-correspondable. A graph $G$ is \textit{decomposable} into graphs $H_1$ and $H_2$ if $H_1$ and $H_2$ are spanning subgraphs of $G$ that partition $E(G)$. The authors of \cite{cho2022decomposing} prove that every planar graph $G$ is decomposable into $H_1$ and $H_2$ such that $H_1$ is 3-degenerate and $H_2$ has maximum degree 2 \cite[Theorem 1.2]{cho2022decomposing}.
As $d$-degenerate graphs are $(d+1)$-correspondable (\cite{bernshteyn2023weak}, it follows $H_1$ is $4$-correspondable; furthermore, as $H_2$ has maximum degree 2, it follows any $\mathcal{H}$-coloring of $G$ for $4$-fold correspondence cover $\mathcal{H}$ will be at most $2$-defective. Thus $G$ is $2$-def $4$-correspondable. 

Similarly, the authors of \cite{cho2022decomposing} show every planar graph can be decomposed into $H_1$ and $H_2$ such that $H_1$ is 2-degenerate and $H_2$ has maximum degree 6, from which it follows every planar graph is $6$-def $3$-correspondable \cite[Theorem 1.6]{cho2022decomposing}. On the other hand, they construct a planar graph that is \textit{not} decomposable into $H_1$ and $H_2$ with $H_1$ being 2-degenerate and $H_2$ having maximum degree 3 \cite[Proposition 1.4]{cho2022decomposing}.

\subsection{Main Results}
Our first result is the following:

\begin{theorem}\label{Theorem: not 3 def 3 corr}
    There exists a planar graph that is not $3$-def $3$-correspondable.
\end{theorem}

This implies the result mentioned above of Cho--Choi--Kim--Park--Shan--Zhu that there exists a graph not decomposable into $H_1$ and $H_2$ with $H_1$ being 2-degenerate and $H_2$ having maximum degree 3 \cite[Proposition 1.4]{cho2022decomposing}. Furthermore, our result together with Theorem 1.3 of \cite{cho2022decomposing} implies that the minimum $d$ such that all planar graphs are $d$-def $3$-colorable is between 4 and 6. 

Our next result concerns the relationship between proper and defective correspondence coloring. We show that Ma--Xu--Zhu's result that there exists a $1$-def $3$-choosable planar graphs that is not $4$-choosable extends to correspondence coloring:

\begin{theorem}\label{theorem: 1d3corr but not 4corr}
There exists a planar graph that is $1$-def $3$-correspondable, but not $4$-correspondable.
\end{theorem}

While the construction used is similar to that of Ma, Xu, and Zhu, the proof is considerably more involved. Indeed, a large portion of this manuscript is dedicated to proving this theorem.

Along these same lines, we note that the graph used to prove Theorem \ref{Theorem: not 3 def 3 corr} is $4$-correspondable, and thus:

\begin{corollary}
    There exists a planar graph that is $4$-correspondable but not $3$-def $3$-correspondable.
\end{corollary}

Finally, we prove results about \textit{outerplanar graphs}.
Recall a planar graph is \textit{outerplanar} if it can be embedded into the plane so that all vertices lie on the boundary of the outer face. Cowen--Cowen--Woodall \cite{cowen1986defective} completely characterized the pairs $(d,k)$ such that all outerplanar graphs are $d$-def $k$-colorable. Namely, they proved all planar graphs are $2$-def $2$-colorable, and that this cannot be improved to 1 defect. \v{S}krekovski \cite{vskrekovski1999list} and independently Eaton--Hull \cite{eaton1999defective} generalized this to list coloring, demonstrating that all outerplanar graphs are $2$-def $2$-choosable. We prove the analogous result for correspondence coloring:

\begin{theorem}
    \label{theorem: outerplanar graph is 3def 2corr}
    All outerplanar graphs are $3$-def $2$-correspondable. Furthermore, 3 defects is best possible.
\end{theorem}



We finish this section with an outline of the rest of this paper. In \S\ref{sec: correspondence}, we give an overview of list and correspondence coloring, with necessary definitions. The reader familiar with correspondence coloring may skip this section (but we note that we use a definition of correspondence cover in which each list of a vertex forms an independent set, rather than a clique; this allows us to draw diagrams cleaner). In \S\ref{sec: bounds on d}, we prove Theorem \ref{Theorem: not 3 def 3 corr}. In \S\ref{section: outerplanar}, we consider outerplanar graphs, proving Theorem \ref{theorem: outerplanar graph is 3def 2corr}. A substantial amount of work goes into proving Theorem \ref{theorem: 1d3corr but not 4corr}, for which we reserve \S\S\ref{section: proof of 1def3corr but not 4corr}--\ref{Section: T(8) not $4$-correspondable}. Finally, in \S\ref{section: open}, we  discuss further directions and open problems.

\section{List and Correspondence Coloring}\label{sec: correspondence}
In this section, we give an overview of list and correspondence coloring. A reader familiar with these definitions may omit this section; though we do caution the reader that we use a slightly different definition of correspondence coloring than usual, in that we do not form cliques between the colors in the list of a vertex, and instead leave these as independent sets. Much of this text is taken directly from \cite{anderson2024coloring}. 

Introduced independently by Vizing \cite{vizing1976coloring} and Erd\Hungarian{o}s--Rubin--Taylor \cite{erdos1979choosability}, \textit{list coloring} is a generalization of graph coloring in which each vertex is assigned a list of colors from which it may be colored. Formally, we say $L : V(G) \to 2^{\N}$ is a \textit{list assignment} for $G$, and an \textit{$L$-coloring} of $G$ is a proper coloring $\phi: V(G) \to \N$ such that $\phi(v) \in L(v)$ for each $v \in V(G)$. When $|L(v)| \geq k$ for each $v \in V(G)$, where $k \in \N$, we say $L$ is \textit{$k$-fold}. The \textit{list chromatic number} of $G$, denoted $\chi_{\ell}(G)$, is the smallest $k$ such that $G$ has an $L$-coloring for every $k$-fold list assignment $L$ for $G$. When $\chi_{\ell} \leq k$, we say $G$ is \textit{$k$-choosable} or $k$ list-colorable.

It is often convenient to view list coloring from a different perspective. Given a graph $G$ and a list assignment $L$ for $G$, we create an auxiliary graph $H$ as follows: 
\begin{align*}
    V(H) &\defeq \{(v, c) \in V(G) \times \N : c \in L(v)\} \\
    E(H) &\defeq \{\{(v, c), (u, d)\} : vu \in E(G),\, c = d\}\}.
\end{align*}
We call $H$ a \textit{cover graph} of $G$, and the pair $(L, H)$ a \textit{list cover} of $G$. An $L$-coloring of $G$ is then an independent set $I$ in $H$ which satisfies $|I| = |V(G)|$, and, for each vertex $v \in V(G)$, there exists $c \in \N$ such that $(v, c) \in I$. Thus $I$ selects exactly one vertex of $H$ of the form $(v,c)$ for each $v \in V(G)$.

\emphdef{Correspondence coloring} (also known as \emphdef{DP-coloring}) is a generalization of list coloring introduced by Dvo\v{r}\'ak and Postle \cite{dvovrak2018correspondence} in order to solve a question of 
Borodin. Just as in list coloring, each vertex is assigned a list of colors, $L(v)$;
in contrast to list coloring, though, the identifications between the colors in the lists are allowed to vary from edge to edge.
That is, each edge $uv \in E(G)$ is assigned a matching $M_{uv}$ (not necessarily perfect, not necessarily the identity, and possibly empty) from $L(u)$ to $L(v)$.
A proper correspondence coloring is then a mapping $\phi : V(G) \to \N$ satisfying $\phi(v) \in L(v)$ for each $v \in V(G)$ and $\phi(u)\phi(v) \notin M_{uv}$ for each $uv \in E(G)$.
Formally, correspondence colorings are defined in terms of an auxiliary graph known as a \emphdef{correspondence cover} of $G$.

\begin{definition}[Correspondence Cover]\label{def:corr_cov}
    A \emphdef{correspondence cover} of a graph $G$ is a pair $\mathcal{H} = (L, H)$, where $H$ is a graph and $L \,:\,V(G) \to 2^{V(H)}$ such that:
    \begin{enumerate}[label= ({\arabic*}), leftmargin = \leftmargin + 1\parindent]
        \item The set $\{L(v)\,:\, v\in V(G)\}$ forms a partition of $V(H)$,
        \item\label{dp:list_independent} For each $v \in V(G)$, $L(v)$ is an independent set in $H$, and
        \item\label{dp:matching} For each $u, v\in V(G)$, the edge set of $H[L(u) \cup L(v)]$ forms a matching (which we call $M_{uv}$), which is empty if $uv \notin E(G)$.
    \end{enumerate}
    We refer to $(3)$ as the \textit{matching condition}.
\end{definition}
We call the vertices of $H$ \emphdef{colors}.
For $c \in V(H)$, we let $L^{-1}(c)$ denote the \emphdef{underlying vertex} of $c$ in $G$, i.e., the unique vertex $v \in V(G)$ such that $c \in L(v)$. Given a list assignment $L$ of a graph $G$, we let $\ell(v) \defeq |L(v)|$ for $v \in V(G)$. If two colors $c$, $c' \in V(H)$ are adjacent in $H$, we say that they \emphdef{correspond} to each other, or that they \textit{conflict} with each other, and write $c \sim_H c'$, or simply $c \sim c'$ when $H$ is clear from context. Instead of writing $c_1 \sim c_2$ and $c_2 \sim c_3$, we will often write $c_1 \sim c_2 \sim c_3$ (but note it is still possible $c_1 \not\sim c_3$.)

Many times we will abuse notation and create a correspondence cover $(L, H)$ for a graph $G$ by using the same label for different vertices of $H$. For example, we may let $\H$ be the correspondence cover for $G$ by taking $L(v) = \{1, 2, 3\}$ for all $v \in V(G)$ (as opposed to something clumsier such as $L(v) = \{(v,1), (v,2), (v,3)\}$.) In this way, we view a correspondence cover for $G$ as a list assignment for $G$ where the lists may use the same colors, but in which colors conflict differently based on each matching $M_e$ for $e \in E(G)$.

An \emphdef{$\mathcal{H}$-coloring} is a mapping $\phi \colon V(G) \to V(H)$ such that $\phi(v) \in L(v)$ for all $v \in V(G)$. Similarly, a \emphdef{partial $\mathcal{H}$-coloring} is a partial mapping $\phi \colon V(G) \dasharrow V(H)$ such that $\phi(v) \in L(v)$ whenever $\phi(v)$ is defined.
A (partial) $\mathcal{H}$-coloring $\phi$ is \emphdef{proper} if the image of $\phi$ is an independent set in $H$, i.e., if $\phi(u) \not \sim \phi(v)$
for all $u$, $v \in V(G)$ such that $\phi(u)$ and $\phi(v)$ are both defined. Notice, then, that the image of a proper $\mathcal{H}$-coloring of $G$ is exactly an independent set $I \subseteq V(H)$ with $|I \cap L(v)| = 1$ for each $v \in V(G)$.

A correspondence cover $\mathcal{H} = (L,H)$ is \emphdef{$k$-fold} if $|L(v)| \geq k$ for all $v \in V(G)$. The \emphdef{correspondence chromatic number} of $G$, denoted by $\chi_{c}(G)$, is the smallest $k$ such that $G$ admits a proper $\mathcal{H}$-coloring for every $k$-fold correspondence cover $\mathcal{H}$. If $\chi_{c}(G) \leq k$, then we say $G$ is \textit{$k$-correspondable}.

Note that a list cover $(L, H)$ of $G$ is a correspondence cover for $G$, where each matching $M_{uv}$ is such that $(u,c)$ matches with $(v,c)$ for each $c \in \N$. As classical coloring is the special case of list coloring in which all lists are identical, it follows that $\chi(G) \leq \chi_{\ell}(G) \leq \chi_{c}(G)$. We also note that the following lemma which follows from a greedy coloring (see \cite{bernshteyn2023weak} for a formal proof:)

\begin{lemma}\label{greedy}
    If $G$ is $d$-degenerate, then it is $d+1$-correspondable.
\end{lemma}

We now provide a few definitions:

\begin{definition}
    Let $G_1$ and $G_2$ be isomorphic (say by $\sigma$) and let $\H_1 = (L_1, H_1)$ and $\H_2 = (L_2, H_2)$ be correspondence covers of $G_1$ and $G_2$ respectively. Then we say $\H_1$ and $\H_2$ are \emphdef{isomorphic}, denoted $\H_1 \cong \H_2$, if there exists an isomorphism $\psi$ from $H_1$ to $H_2$ that preserves lists, i.e., for all $v \in G_1$ and $c \in H_1$ it follows $c \in L_1(v) \iff \psi(c) \in L_2(\sigma(v))$.
\end{definition}

\begin{definition}
    Let $G$ be a graph and $\H$ a correspondence cover for $G$. Suppose $\phi$ is a partial $\mathcal{H}$-coloring of $G$. Define the following:
    \begin{enumerate}
        \item For each $v \in G \setminus \dom(\phi)$, let $L_\phi(v) \defeq \{c \in L(v) : \forall\, c' \in \im(\phi)\, (c \notsim c')\}$.
        \item $H_\phi \defeq H[\bigcup_{v \in G \setminus \dom(\phi)}L_\phi(v)]$.
    \end{enumerate}
    Then the correspondence cover $\H_\phi \defeq (L_\phi, H_\phi)$ of $G \setminus \dom(\phi)$ is the \emphdef{subcover of $\H$ induced by $\phi$}.
\end{definition}

\begin{remark}
\label{remark: simplify cover}
    We make two observations about (defective) correspondence coloring that allow us to simplify various proofs. First, suppose $\H = (L, H)$ is a $k$-fold cover for $G$ for $k \in \N$. Then if $\H^*$ is a correspondence cover for $G$ formed by removing colors from $H$ so that $|L(v)| = k$ for each $v \in V(G)$, then a $d$-def $\H^*$-coloring is a $d$-def $\H$-coloring. Second, if $\H^*$ is a correspondence cover for $G$ formed by adding edges to $\H$ so that every matching $M_e$ for $e \in E(G)$ is maximal, then a $d$-def $\H^*$ coloring is a $d$-def $\H$-coloring. Thus there will usually be no harm in assuming $\H$ has list sizes exactly $k$ and that all matchings are maximal.
\end{remark}

\section{Proof of Theorem \ref{Theorem: not 3 def 3 corr}} \label{sec: bounds on d}
In this section, we prove Theorem \ref{Theorem: not 3 def 3 corr} by constructing a planar graph $G$ that is not $3$-def $3$-correspondable. First, we make some preliminary definitions and prove a lemma.

Let $T$ and $H$ be the graphs depicted in Fig.~\ref{fig: GadgetTforNot3Def3Corr} and Fig.~\ref{fig: GadgetCoverNot3Def3Corr}, respectively.
\begin{figure}
    \centering
    \includegraphics[scale = .7]{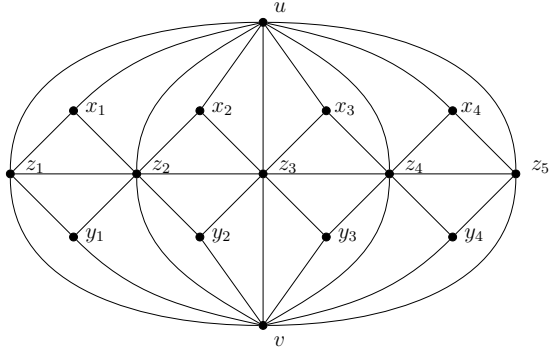}
    \caption{Graph $T$ in the proof of Theorem ~\ref{Theorem: not 3 def 3 corr}.}
    \label{fig: GadgetTforNot3Def3Corr}
\end{figure}
\begin{figure}
    \centering
    \includegraphics[scale = .7]{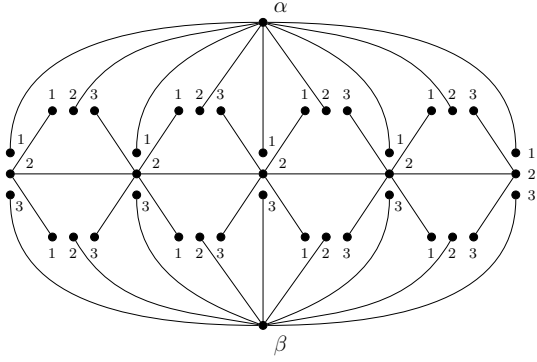}
    \caption{Graph $H$; a cover for graph $T$ in the proof of Theorem~\ref{Theorem: not 3 def 3 corr}.}
    \label{fig: GadgetCoverNot3Def3Corr}
\end{figure}
Let $L : V(G) \to 2^{V(H)}$ by $L(u) = \{\alpha\}$, $L(v) = \{\beta\}$, and $L(x) = \{1, 2, 3\}$ for all $x \in V(G)\setminus\{u, v\}$. Then $\H \defeq (L, H)$ is a correspondence cover for $T$; we call such a cover \textit{bad} for the pair $(\alpha, \beta)$.

\begin{lemma}\label{lemma: bad cover for 3def3corr}
     Let $\H$ be a correspondence cover for $T$ that is bad for $(\alpha, \beta)$. Then any $3$-defective $\H$-coloring of $T$ satisfies $\defect(u) = 1$ or $\defect(v) = 1$.
\end{lemma}
\begin{proof}
    Suppose $\phi$ is a $3$-def $\H$-coloring of $T$ such that $\defect(u) = \defect(v) = 0$. Then clearly $\phi(u) = \alpha$, $\phi(v) = \beta$, $\phi(z_i) = 2$ for all $1 \leq i \leq 5$, $\phi(x_j) \neq 2$ and $\phi(y_j) \neq 2$ for all $1 \leq j \leq 4$. If $\phi(x_2)= \phi(y_2)  = 1$, then $\defect(z_2) \geq 4$, a contradiction. Similarly, if $\phi(x_2)= \phi(y_2)  = 3$, then $\defect(z_3) \geq 4$, a contradiction. Thus it follows $x_2$ and $y_2$ are not the same color; a similar argument shows $x_3$ and $y_3$ are not the same color. But then it is easy to see that $z_3$ will conflict with at least two of $x_2$, $y_2$, $x_3$, $y_3$, and thus $\defect(z_3) \geq 4$, a contradiction.
    
\end{proof}

We are now ready to prove Theorem \ref{Theorem: not 3 def 3 corr}.

\begin{theorem*}[\ref{Theorem: not 3 def 3 corr}]
    There exists a planar graph that is not $3$-defective $3$-correspondable.
\end{theorem*}

\begin{proof}
Let $G$ consist of 63 copies of $T$ with all the copies of $x$ identified together as a single vertex and all the copies of $y$ identified together as a single vertex. Note $G$ is planar.
Construct a cover $\H = (L, H)$ for $G$ as follows: let $L(v) = \{1, 2, 3\}$ for every $v \in V(G)$. Now for each of the 9 pairs $(\alpha, \beta) \in L(x) \times L(y)$, let 7 copies of $T$ be bad for $(\alpha, \beta)$. 

Now suppose $\phi$ is a $3$-def $\H$-coloring of $G$. By construction, there are 7 copies of $T$ that are bad for $(\phi(x), \phi(y))$. By Lemma \ref{lemma: bad cover for 3def3corr}, each bad copy of $T$ provides one defect to either $x$ or $y$, and thus $\defect(x) \geq 4$ or $\defect(y) \geq 4$; this contradicts $\phi$ is a $3$-defective $\H$-coloring. Thus $G$ is a planar graph that is not $3$-defective $3$-correspondable. Furthermore, it is easy to see that $G$ is $3$-degenerate, and thus it is $4$-correspondable.
\end{proof}


\section{Outerplanar graphs}\label{section: outerplanar}
In this section, we prove Theorem \ref{theorem: outerplanar graph is 3def 2corr}. To do so, we prove a slightly stronger theorem below. Recall an outerplanar graph is a \textit{near triangulation} if every interior face is a triangle.
\begin{theorem}\label{Theorem: outerplanar}
If $G$ is an outer planar near triangulation with outer cycle $C$, $u,v \in V(C)$ with $uv \in E(C)$, $\H$ is a 2-fold correspondence cover for $G$, and $\phi$ is a partial $3$-def $\H$-coloring of $G$ with domain $u, v$, then $\phi$ can be extended to a $3$-def $\H$-coloring of $G$ such that:
\begin{enumerate}
    \item $\defect(u) \leq 1$ and $\defect(v) \leq 2$ if $\phi(u) \sim \phi(v)$,
    \item $\defect(u) \leq 0$ and $\defect(v) \leq 1$ if $\phi(u) \not\sim \phi(v)$.
\end{enumerate}
\end{theorem}

\begin{proof}
    We proceed by induction on the number of vertices of $G$. The base case of $G$ having 3 vertices is trivial. 
    So assume $|V(G)| \geq 4$. We split into cases based on $\deg(u)$.
    
    Case 1: $\deg(u) = 2$.
    Let $x$ be the neighbor of $u$ not equal to $v$. Then, as $G$ is a near triangulation, it follows $xv \in E(G)$. 
    Let $\psi$ be a partial $\H$-coloring with domain $\{x,v\}$ by taking $\psi(v) \defeq \phi(v)$ and $\psi(x)$ to be the color in $L(x)$ that doesn't conflict with $\phi(u)$. 
    By the induction hypothesis applied to $G\setminus \{u\}$ and $\psi$, we can extend $\psi$ to a $3$-def $\H$-coloring of $G\setminus \{u\}$ such that $\defect_{\psi}(x) \leq 2$ and $\defect_{\psi}(v)  \leq 1$. 
    Now extend $\psi$ by taking $\psi(u) \defeq \phi(u)$. If $\phi(u) \sim \phi(v)$, then $\defect_\psi(u) \leq 1$ and $\defect_\psi(v) \leq 2$; otherwise, $\defect_\psi(u) = 0$ and $\defect_\psi(v) \leq 1$. Thus $\psi$ is the desired extension of $\phi$, completing the induction.
    
    Case 2: $\deg(u) > 2$. By symmetry, we may assume $\deg(v) > 2$, as otherwise we can apply case 1 to $v$. 
    Starting at $v$ and preceding counter-clockwise along $C$, assume the vertices along $C$ are $v, u, x_1, x_2, \ldots, x_k, v$ for some $k \in \N$. Since $\deg(u), \deg(v) \geq 3$, and as $G$ is a near-triangulation, it follows $u$ has a neighbor $x_i$ for some $i$; let $i^*$ be the largest $i$ such that $u$ is adjacent to $x_{i^*}$, and set $z \defeq x_{i^*}$. Since $u$ is not adjacent to $x_{i^*+1}$, and as $G$ is a near triangulation, it follows $v \sim z$. Let $G_1$ be the graph induced by $u, x_1, \ldots, z$, and $G_2$ be the graph induced by $v, x_k, x_{k-1}, \ldots, z$. Take $\psi_1$ to be a partial $\H$-coloring with domain $\{u, z\}$ by taking $\psi_1(u) \defeq \phi(u)$ and $\psi_1(z)$ to be the color in $L(z)$ that doesn't conflict with $\phi(u)$. Define $\psi_2$ to be a partial $\H$-coloring with domain $\{v,z\}$ with $\psi_2(v) \defeq \phi(v)$ and $\psi_2(z) \defeq \psi_1(z)$.
    By the induction hypothesis, $\psi_1$ can be extended to a $3$-def $\H$-coloring of $G_1$ with $\defect_{\psi_1}(u) \leq 0$ and $\defect_{\psi_1}(z) \leq 1$. Similarly, by the induction hypothesis, $\phi_2$ can be extended to a $3$-def $\H$-coloring of $G_2$ with $\psi_2(v) \leq 1$ and $\psi_2(z) \leq 2$. It follows $\psi_1 \cup \psi_2$ is a $3$-def $\H$-coloring of $G$ with the desired properties.
\end{proof}

We now prove Theorem \ref{theorem: outerplanar graph is 3def 2corr}.

\begin{theorem*}[\ref{theorem: outerplanar graph is 3def 2corr}]
    All outerplanar graphs are $3$-def $2$-correspondable. Furthermore, 3 defects is best possible.
\end{theorem*}
\begin{proof}
As graphs with fewer than 3 vertices are clearly $3$-def $2$-correspondable, we focus on outerplanar graphs with at least 3 vertices.
Since edges can always be added to an outerplanar graph with at least 3 vertices to make it a near triangulation, it follows from Theorem \ref{Theorem: outerplanar} that such outerplanar graphs are $3$-def $2$-correspondable (to see this, first choose any edge $uv$ on the outerface. Then choose a partial $\mathcal{H}$-coloring $\phi$ by selecting values for $\phi(u)$ and $\phi(v)$. By Theorem \ref{Theorem: outerplanar}, $\phi$ can be extended to all of $G$.)

We now prove that 3 defects is best possible.

\begin{figure}[h!]
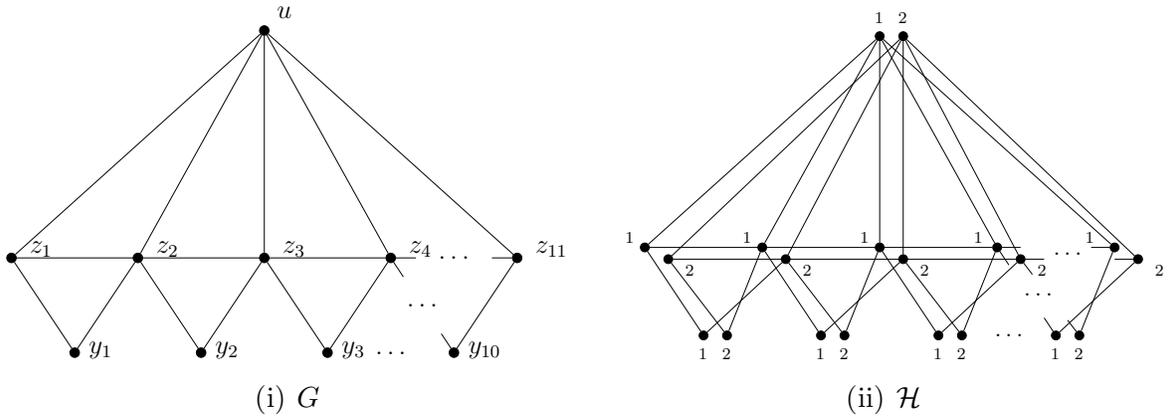

    \centering
    \begin{subfigure}[t]{.48\textwidth}
     \centering
    \includegraphics[width=1\textwidth]{figures/Figure_OuterplanarTfor2Def2Col.tex}
    \caption{$G$}
    \label{fig: outerplanar T}
\end{subfigure}%
\begin{subfigure}[t]{.48\textwidth}
   \centering
    \includegraphics[width=1\textwidth]{figures/Figure_OuterplanarCover2Def2Corr.tex}
    \caption{$\H$}
    \label{fig: outerplanar H}
\end{subfigure}
\caption{$G$ and correspondence cover $\H$ in proof of Theorem \ref{theorem: outerplanar graph is 3def 2corr}.}
\end{figure}



    Consider the fan of length 12, i.e., a path $z_1, z_2, \ldots, z_{12}$ with a vertex $u$ adjacent to each $z_i$. Now for each edge $z_iz_{i+1}$, $i \in [11]$, add a vertex $y_i$ adjacent to $z_i$ and $z_{i+1}$. Call the resulting graph $G$ (see Fig \ref{fig: outerplanar T}). Clearly $G$ is outerplanar.
    
    Let $\H$ be a correspondence cover for $T$ as follows: for each $i \in [11]$, let $M_{y_iz_{i+1}}$ be the matching with 1 corresponding to 2 and vice-versa. Let all other matchings be the identity matching (see Figure \ref{fig: outerplanar H}).

    Assume for contradiction that $\phi$ is a $2$-def $\H$-coloring of $G$. Then there are at most two indices, say $i_1$ and $i_2$, such that $\phi(u) \sim  \phi(z_{i_1}) \sim \phi(z_{i_2})$. Thus there exists $j \in [9]$ such that $z_{j}$, $z_{j+1}$, $z_{j+2}$, $z_{j+3}$ all receive different colors from $u$, and thus are all colored the same. For ease of notation, we will assume without loss of generality that $j = 1$ and $z_1, \ldots, z_4$ are all colored 1. If $\phi(y_2) = 1$, then $\defect(z_2) = 3$, a contradiction. If $\phi(y_2) = 2$, then $\defect(z_3) = 3$, again a contradiction. Thus such a $\phi$ cannot exist. This completes the proof.


\end{proof}

\section{Proof of Theorem \ref{theorem: 1d3corr but not 4corr}}\label{section: proof of 1def3corr but not 4corr}
In this section, we prove Theorem \ref{theorem: 1d3corr but not 4corr}, modulo some technical lemmas whose proofs we postpone to later sections.

By Remark \ref{remark: simplify cover}, we will assume for all $k$-fold correspondence covers in this section, all lists have size exactly $k$ and all matchings are maximal. We introduce some notation. Let $R$ be the graph isomorphic to $K_4$ minus an edge given by $V(R) =  \{a, b, c, d\}$ and all edges except $ac$ (see Figure \ref{Fig: R}). Let $T$ be the graph displayed in Figure \ref{Fig: T}. Let $R_1$ be the subgraph $T[x, u_1, v_1, z]$, and $R_2$ be the subgraph $T[z, u_2, v_2, y]$.

\begin{figure}[h!]
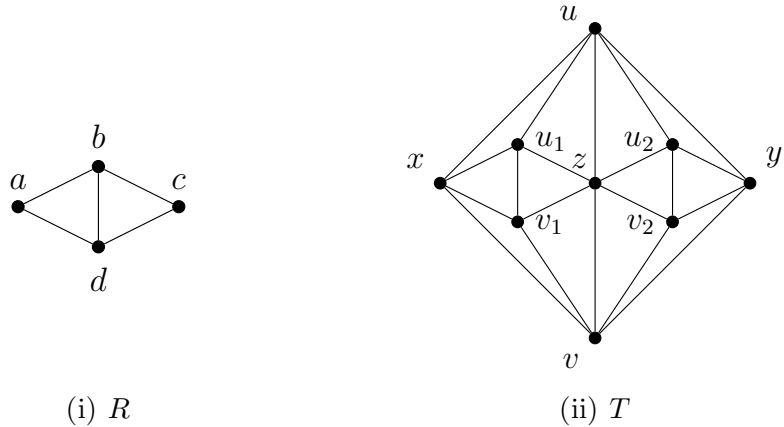

\centering
\begin{subfigure}[t]{.4\textwidth}
     \centering
    \includegraphics[width=.4\textwidth]{figures/Figure_R.tex}
    \caption{$R$}
    \label{Fig: R}
\end{subfigure}%
\begin{subfigure}[t]{.4\textwidth}
   \centering
    \includegraphics[width=.8\textwidth]{figures/Figure_T.tex}
    \caption{$T$}
    \label{Fig: T}
\end{subfigure}
\caption{Graphs $R$ and $T$.}
\end{figure}

\subsection{Twisted and Wedged}
\begin{definition}\label{def: twisted and wedged}
    Let $\mathcal{H} = (L, H)$ be a correspondence cover for $R$ with $L(a) = \{a_1\}$, $L(b) = \{b_1, b_2\}$, $L(c) = \{c_1\}$, $L(d) = \{d_1, d_2\}$. Then we say $\mathcal{H}$ is
    \begin{enumerate}
        \item \emphdef{twisted} if $a_1b_1d_2c_1b_2d_1a_1$ is a cycle in $H$,
        \item \emphdef{wedged} if $a_1b_1$, $a_1d_1$, and at least two of $\{b_2d_2, c_1b_2, c_1d_2\}$ are edges in $H$.
    \end{enumerate}
    If $G$ is a graph isomorphic to $R$ and $\H'$ is a correspondence cover for $G$, then we say $\H'$ is \emphdef{twisted} (respectively, \emphdef{wedged}) if $\H'\cong \H$ and $\H$ is twisted (respectively, wedged). Note that if $\H$ is twisted, then $\H$ is wedged. (See Figure \ref{Fig: Twisted} and Figure \ref{Fig: Wedged}.)
\end{definition}

\begin{figure}[h]
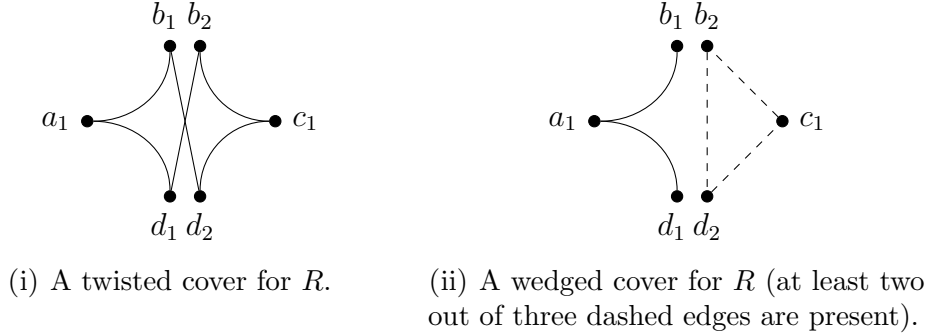

\centering
\begin{subfigure}[t]{.4\textwidth}
     \centering
    \includegraphics[scale = 1]{figures/Figure_Twisted.tex}
    \caption{A twisted cover for $R$.}
    \label{Fig: Twisted}
\end{subfigure}%
~
\begin{subfigure}[t]{.4\textwidth}
   \centering
    \includegraphics[scale = 1]{figures/Figure_Wedged.tex}
    \caption{A wedged cover for $R$ (at least two out of three dashed edges are present).}
    \label{Fig: Wedged}
\end{subfigure}
\caption{Twisted and wedged covers for $R$.}
\end{figure}

\begin{lemma}\label{Lemma: 1-1 not twisted or wedged}
    Let $\H$ be a correspondence cover for $R$ with $\ell(a) = \ell(c) = 1$ and $\ell(b) = \ell(d) = 2$.
    If $\mathcal{H}$ is not twisted, then there exists a $1$-def $\mathcal{H}$-coloring $\phi$. Furthermore, if $H$ is not wedged, then we may take $\phi$ such that $\defect_\phi(a) = 0$.
\end{lemma}
\begin{proof}
By applying an appropriate isomorphism to $\H$, we may assume without loss of generality that $\H = (L, H)$ is such that $L(a) = \{a_1\}$, $L(b) = \{b_1, b_2\}$, $L(c) = \{c_1\}$, $L(d) = \{d_1, d_2\}$. By Remark \ref{remark: simplify cover}, all matchings are maximal, and thus we may assume by relabeling colors if necessary that $a_1 \sim b_1$ and $a_1 \sim d_1$. Set $\phi(a) = a_1$ and $\phi(c) = c_1$.

Case 1: $a_1$ and $c_1$ share a common neighbor. 
Without loss of generality, assume the common neighbor is $b_1$. Set $\phi(b) = b_2$. 
If $d_2 \notsim b_2$, or $d_2 \notsim c_1$, then set $\phi(d) = d_2$ and observe $\phi$ is a $1$-defective $\H$-coloring with $\defect_\phi(a) = 0$.
Otherwise, $d_2 \sim b_2$ and $d_2 \sim c_1$, and thus $\H$ is wedged. By the matching condition it follows $d_1 \notsim b_2$ and $d_1 \notsim c_1$. Thus setting $\phi(d) = d_1$ results in a $1$-defective $\H$-coloring of $R$.

Case 2: $a_1$ and $c_1$ do not share a common neighbor. Then $c_1 \sim b_2$, and $c_1 \sim d_2$. Note that $\H$ is wedged. If $b_1 \sim d_2$ and $d_1 \sim b_2$, then $\H$ is twisted, thus we may assume instead that $b_1 \sim d_1$ and $b_2 \sim d_2$. Then setting $\phi(b) = b_1$, $\phi(d) = d_2$ yields a $1$-defective $\H$-coloring of $R$.
\end{proof}

\begin{definition}\label{wedged 2}
    Let $\mathcal{H} = (L, H)$ be a correspondence cover for $R$ with $L(a) = \{a_1\}$, $L(b) = \{b_1, b_2\}$, $L(c) = \{c_1, c_2\}$, $L(d) = \{d_1, d_2\}$. Then we say $\mathcal{H}$ is \emphdef{wedged} if $b_1\sim a_1 \sim d_1$ and $c_1 \sim b_2 \sim d_2 \sim c_2$.
    If $G$ is a graph isomorphic to $R$ and $\H'$ is a correspondence cover for $G$, then we say $\H'$ is \emphdef{wedged} if $\H' \cong \H$. (See Figure \ref{Figure: wedged2})
\end{definition}

\begin{figure}[h!]
     \centering
    \includegraphics[scale = 1]{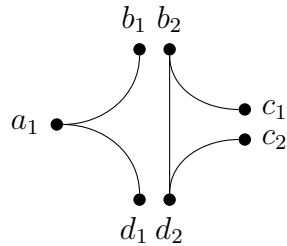}
    \caption{A wedged cover for $R$ with $|L(c)| = 2$.}
    \label{Figure: wedged2}
\end{figure}

Note that although we are defining \textit{wedged} in both Definition \ref{def: twisted and wedged} and Definition \ref{wedged 2}, the first is for covers $\H$ of $R$ with $|L(c)| = 1$, and the second for covers $\H$ of $R$ with $|L(c)| = 2$. It should be clear from context how the definition is being applied.

\begin{lemma}\label{Lemma: 1-2}
    Let $\H = (L, H)$ be a correspondence cover for $R$ with $\ell(a) = 1$ and $\ell(b) = \ell(c) = \ell(d) = 2$.
        \begin{enumerate}[label = {(\roman{*})}]
        \item\label{Lemma: 1-2 coloring} There exists a $1$-def $\H$-coloring $\phi$ such that $\defect_\phi(c) = 0$.
        \item\label{Lemma: 1-2 not wedged} If $\H$ is not wedged, then there exists a $1$-def $\H$-coloring $\phi$ such that $\defect_\phi(a) = 0$ (though $\defect_\phi(c)$ may equal 1).
            
        \end{enumerate} 
\end{lemma}
\begin{proof}
    Without loss of generality, we may assume $\H = (L, H)$ is such that $L(a) = \{a_1\}$, $L(b) = \{b_1, b_2\}$, $L(c) = \{c_1, c_2\}$, $L(d) = \{d_1, d_2\}$, $a_1 \sim b_1$ and $a_1 \sim d_1$. Set $\phi(a) = a_1$. 

    Case 1: $H$ contains the cycle $a_1b_1c_id_1a_1$ for some $i \in \{1, 2\}$. Without loss of generality, we may assume $i = 1$, and take $\phi(b) = b_2$, $\phi(c) = c_1$, and $\phi(d) = d_2$. 
    Then $\phi$ is a $1$-defective $\H$-coloring of $R$ with $\defect_\phi(c) = 0$. Note as well that in this case, $\H$ is not wedged, and $\defect_\phi(a) = 0$, as desired.

    Case 2: $H$ does not contain a cycle $a_1b_1c_id_1a_1$ for some $i \in \{1, 2\}$. Then we may assume $d_1 \sim c_2 \sim b_2$ or $d_2 \sim c_2 \sim b_1$; without loss of generality, assume the former. Then $d_2 \sim c_1 \sim b_1$.
    
    If $b_1 \sim d_1$, then $b_2 \sim d_2$; thus $\H$ is wedged (under the isomorphism switching $c_1$ and $c_2$). Take $\phi(b) = b_1$, $\phi(c) = c_2$, and $\phi(d) = d_2$. Observe $\phi$ is a $1$-defective $\H$-coloring of $R$ with $\defect_{\phi}(c) = 0$. Otherwise, $b_1 \sim d_2$, and thus $d_1 \sim b_2$. Take $\phi(b) = b_1$, $\phi(c) = c_2$, $\phi(d) = d_2$. Then $\phi$ is a $1$-def $\H$-coloring with $\defect_\phi(c) = 0$. On the other hand, taking instead $\phi(b) = b_2$, $\phi(c) = c_2$, and $\phi(d) = d_2$ results in a $1$-defective $\H$-coloring with $\defect_\phi(a) = 0$.
\end{proof}

\begin{lemma}\label{lemma: not twisted and wedged}
    Let $\H$ be a correspondence cover for $R$ with $\ell(a) = 1$ and $\ell(b) = \ell(c) = \ell(d) = 2$. Suppose $L(c) = \{c_1, c_2\}$ and $\H\setminus\{c_2\}$ is twisted. Then $\H\setminus\{c_1\}$ is not wedged; furthermore, there is $0$-def $\H$-coloring $\phi$ of $R$ with $\phi(c) = c_2$.
\end{lemma}
\begin{proof}
    It suffices to prove the furthermore part, as it is easy to see the existence of such a coloring implies $\H\setminus\{c_1\}$ is not wedged.
    Suppose $\H\setminus\{c_2\}$ is twisted. Without loss of generality, we may assume $\H$ is such that $L(a) = \{a_1\}$, $L(b) = \{b_1, b_2\}$, and $L(d) = \{d_1, d_2\}$ and that $a_1b_1d_2c_1b_2d_1a_1$ is a cycle in $H$. Then $b_2 \notsim c_2$, $d_2 \notsim c_2$, $b_2 \notsim d_2$, $b_2 \notsim a_1$, $d_2 \notsim a_1$. Thus $\phi(a) = a_1$, $\phi(b) = b_2$, $\phi(c) = c_2$, and $\phi(d) = d_2$ is a $0$-def $\H$-coloring of $R$.
\end{proof}

\subsection{Good and Bad}
We introduce some notation that will allow us to better classify certain correspondence covers of $T$.

\begin{definition}\label{def: good and bad}
    Let $\H$ be a correspondence cover for $T \setminus \{u, v\}$ with $\ell(w) = 2$ for all $w \in \{u_1, u_2, v_1, v_2\}$ and $1 \leq \ell(x), \ell(y), \ell(z) \leq 2$. We say $\H$ is \emphdef{bad} if one of the following occurs:
    \begin{enumerate}[label = {(\arabic*-bad)}, itemsep = .5em, leftmargin = \leftmargin+2\parindent]
        \item $\ell(x) = 2$, $\ell(z) = 1$, $\ell(y) = 2$, and $\H(R_1)$ and $\H(R_2)$ are both wedged.
        \item $\ell(x) = 1$, $\ell(z) = 1$, $\ell(y) = 2$, and \begin{enumerate}[label = {(\roman*)}]
            \item $\H(R_1)$ is twisted, or
            \item $\H(R_1)$ and $\H(R_2)$ are wedged.
        \end{enumerate}
        \item $\ell(x) = 2$, $\ell(z) = 1$, $\ell(y) = 1$, and \begin{enumerate}[label = {(\roman*)}]
            \item $\H(R_2)$ is twisted, or
            \item $\H(R_1)$ and $\H(R_2)$ are wedged.
        \end{enumerate}
    \item $\ell(x) = 1$, $\ell(z) = 2$, $\ell(y) = 1$, and there exist distinct $z_1, z_2 \in L(z)$ such that $\H(R_1) \setminus \{z_2\}$ and $\H(R_2) \setminus \{z_1\}$ are twisted.
    \item $\ell(x) = 1$, $\ell(z) = 1$, $\ell(y) = 1$, and \begin{enumerate}[label = {(\roman*)}]
            \item $\H(R_1)$ is twisted, or $\H(R_2)$ is twisted, or
            \item $\H(R_1)$ and $\H(R_2)$ are wedged.
        \end{enumerate}
    \end{enumerate}
    We say $\H$ is \emphdef{good} otherwise. 
\end{definition}

Suppose $\H$ is a $3$-fold correspondence cover for $T$, and $\phi$ is a partial $\H$-coloring of $T$ with $\dom(\phi) = \{u, v\}$. Then observe $\mathcal{H}_\phi$ is a correspondence cover for $T\setminus\{u, v\}$ such that $\ell_\phi(w) = 2$ for all $w \in \{u_1, u_2, v_1, v_2\}$ and $1 \leq \ell_\phi(x), \ell_\phi(y), \ell_\phi(z) \leq 2$. 
Thus $\H_{\phi}$ is either good or bad. This justifies the following definition:

\begin{definition}
    Let $\H$ be a 3-fold cover for $T$, and let $(\alpha, \beta) \in L(u) \times L(v)$. Then $T$ is \emphdef{good} (resp. \emphdef{bad}) for $(\alpha, \beta)$ if $\H_\phi$ is good (resp. bad), where $\phi$ is the partial $\H$-coloring of $T$ with domain $\{u, v\}$ such that $\phi(u) = \alpha$ and $\phi(v) = \beta$.
\end{definition}

\begin{lemma}\label{lemma: bad for at most 3}
    Let $\H$ be a 3-fold correspondence cover for $T$. Then $T$ is bad for at most 6 pairs $(\alpha, \beta) \in L(u) \times L(v)$.
\end{lemma}
\begin{proof}
    See \S\ref{Section: not 7 bad pairs}.
\end{proof}

For the next two lemmas, let $\mathcal{H}$ be a 3-fold correspondence cover for $T$, and $\phi$ a partial $\H$-coloring of $T$ with $\dom(\phi) = \{u, v\}$. As discussed above, $\H_\phi$ is either good or bad.

\begin{lemma}\label{Lemma: good}
     If $\H_{\phi}$ is good, then $\phi$ can be extended to a $1$-defective $\H$-coloring $\psi$ of $T$ such that $\defect_\psi(u) = \defect_\psi(v) = 0$.
\end{lemma}
\begin{proof}
    See \S\ref{Section: proof of good}.
\end{proof}

\begin{lemma}\label{lemma: bad}
 If $\H_{\phi}$ is bad, then $\phi$ can be extended to $1$-defective $\H$-colorings $\psi, \rho$ of $T$ such that $\defect_\psi(u) \leq 1$, $\defect_\psi(v) = 0$, and $\defect_\rho(u) = 0$, $\defect_\rho(v) \leq 1$.
\end{lemma}
\begin{proof}
    See \S\ref{section: proof of lemma bad}.
\end{proof}

\subsection{Proof of Theorem \ref{theorem: 1d3corr but not 4corr}}

For an integer $k \in \N$, let $T(k)$ be the graph formed as follows: take $k$ disjoint copies of $T$, and identify $x$ across each copy to a single vertex; similarly, identify $y$ across each copy to a single vertex. The two theorems below show $T(4)$ is $1$-defective $3$-correspondable, but not $4$-correspondable, proving Theorem \ref{theorem: 1d3corr but not 4corr}.

\begin{theorem}\label{Theorem: $1$-def $3$-corr}
    $T(4)$ is $1$-defective $3$-correspondable.
\end{theorem}
\begin{proof}
    Let $\H$ be a 3-fold correspondence cover for $T(4)$. We show there exists $(\alpha, \beta) \in L(u) \times L(v)$ such that $(\alpha, \beta)$ is bad for at most two copies of $T$. Suppose not, and that each pair $(\alpha, \beta) \in T(u) \times T(v)$ is bad for at least 3 copies of $T$. Since Lemma \ref{lemma: bad for at most 3} implies each $T$ is bad for at most $6$ pairs of $(\alpha, \beta)$, a simple double counting argument shows there exists at least 5 copies of $T$, a contradiction. 
    Thus there exists $(\alpha, \beta) \in L(u) \times L(v)$ that is bad for at most 2 copies of $T$. Let $\phi(u) = \alpha$ and $\phi(v) = \beta$. 
    For the two copies of $T$ that are bad for $(\alpha, \beta)$, say $T_i$ and $T_j$, by Lemma \ref{lemma: bad}, we may extend $\phi$ to a $1$-defective $\H$-coloring of $T_i$ such that $\alpha$ will have at most one colored neighbor, and we may extend $\phi$ to a $1$-defective $\H$-coloring of $T_j$ such that $\beta$ will have at most one colored neighbor. Since all other copies of $T$ are good for $(\alpha, \beta)$, by Lemma \ref{Lemma: good} we may extend $\phi$ to a $1$-defective coloring of these copies so $\alpha$ and $\beta$ have no colored neighbors in these copies. The result of these colorings is a $1$-defective $\H$-coloring of $T(4)$.
\end{proof}

\begin{theorem}\label{Theorem: T(8) not $4$-correspondable}
    $T(4)$ is not $4$-correspondable.
\end{theorem}
\begin{proof}
    See \S\ref{Section: T(8) not $4$-correspondable}.
\end{proof}

\section{Proof of Lemma \ref{lemma: bad for at most 3}}\label{Section: not 7 bad pairs}
We introduce some notation that will be useful in this section. Recall $\H = (L, H)$ is a 3-fold correspondence cover for $T$. Motivated by the values of $\ell(x)$, $\ell(z)$, and $\ell(y)$ in Definition \ref{def: good and bad}, we will say $(\alpha, \beta)$ is of \emphdef{type} 2-1 if $(\alpha, \beta)$ is 1-bad or 3-bad for $T$; type 1-1 if $(\alpha, \beta)$ is $2$-bad or $5$-bad for $T$; and type 1-2 if $(\alpha, \beta)$ is $4$-bad for $T$.

\begin{lemma}\label{lemma: not two 1-2's}
Let $(\alpha, \beta)$, $(\alpha, 
\gamma) \in L(u) \times L(v)$ with $\beta \neq \gamma$. Suppose $T$ is bad for $(\alpha, \beta)$ and $(\gamma, \delta)$. Then it is not the case that both $(\alpha, \beta)$ and $(\alpha, \gamma)$ are of type 1-2. It is also not the case that both $(\alpha, \beta)$ and $(\alpha, \gamma)$ are of type 2-1. Furthermore, the analogous statements hold for the pairs $(\beta, \alpha)$, $(\gamma, \alpha)$.
\end{lemma}
\begin{proof}
    Suppose $(\alpha, \beta)$ and $(\alpha, \gamma)$ are of type 1-2. Then $\alpha, \beta$ have a common neighbor $z_1 \in L(z)$, and $\alpha, \gamma$ have a common neighbor $z_2 \in L(z)$. By the matching condition, $z_1 = z_2$ and thus $\beta, \gamma$ have a common neighbor in $L(z)$, contradicting the matching condition. A similar argument shows the 2-1 statement, and the furthermore statement.
\end{proof}

\begin{lemma}\label{lemma: not two twists}
 Let $(\alpha, \beta)$, $(\gamma, \delta) \in L(u) \times L(v)$, with $\alpha, \beta, \gamma, \delta$ all distinct.
   Suppose $T$ is bad for $(\alpha, \beta)$, $(\gamma, \delta)$. Then it is not the case that both $(\alpha, \beta)$ and $(\delta, \gamma)$ are of type 1-2.
\end{lemma}
\begin{proof}
    Suppose $(\alpha, \beta)$, $(\gamma, \delta)$ are of type 1-2. Then $(\alpha, \beta)$ and $(\gamma, \delta)$ are both 4-bad. Let $L(z) = \{z_1, z_2, z_3\}$. Let $\phi$ be the partial $\H$-coloring given by $\phi(u) = \alpha$ and $\phi(v) = \beta$, and $\psi$ be the partial $\H$-coloring with $\psi(u) = \gamma$ and $\psi(v) = \delta$. Since $(\alpha, \beta)$ are type 1-2, they have a common neighbor in $L(z)$; without loss of generality, assume it is $z_1$. Similarly, $\gamma, \delta$ have a common neighbor in $L(z)$. Due to the matching condition, this common neighbor cannot be $z_1$; thus assume without loss of generality it is $z_3$.
    Since $(\alpha, \beta)$ is type 1-2, it follows $\H_{\phi}(R_1)\setminus\{z_2\}$ or $\H_{\phi}(R_2)\setminus\{z_3\}$ is twisted; either way, let $C$ be the cycle witnessing this.
    Now, as $(\gamma, \delta)$ is type 1-2, it follows either $\H_{\psi}(R_1)\setminus\{z_1\}$ or $\H_{\psi}(R_1)\setminus\{z_2\}$ is twisted; either way, let $D$ be the cycle witnessing this. 
    
   We claim $C \neq D$. Indeed, due to the matching condition, the neighbor of $\alpha$ in $L(u_1)$ is not equal to the neighbor of $\gamma$ in $L(u_1)$. Since $C$ uses the two vertices in $L_\phi(u_1)$ and $D$ uses the two vertices in $L_\psi(u_1)$, it follows $C$ and $D$ do not use the same vertices, and thus $C \neq D$.

    Now as $C$ and $D$ both use two vertices in $L(u_1)$, it follows they have a common vertex, say $b$, in $L(u_1)$.  Let $d$ be the neighbor of $b$ in $L(v_1)$. Then as $b$ is a vertex in both $C$ and $D$, it follows both $C$ and $D$ contain the edge $bd$. Thus $d$ is a common vertex in $C$ and $D$. Applying this idea iteratively, it follows that $C$ and $D$ contain all the same edges and thus vertices, and thus $C = D$, contradicting $C \neq D$.
\end{proof}

\begin{lemma}\label{lemma: not all 1's}
Let $(\alpha, \beta)$, $(\alpha, 
\gamma)$, and $(\alpha, \delta) \in L(u) \times L(v)$ with $\beta, \gamma, \delta$ all distinct.
   Suppose $T$ is bad for $(\alpha, \beta)$, $(\alpha, 
\gamma)$, and $(\alpha, \delta)$. Then there does not exist $i, j, k \in \{1, 2\}$ such that $(\alpha, \beta)$, $(\alpha, 
\gamma)$, and $(\alpha, \delta)$ are of type 1-$i$, 1-$j$, and 1-$k$, respectively. Similarly, there does not exist $i, j, k \in 
\{1, 2\}$ such that $(\alpha, \beta)$, $(\alpha, 
\gamma)$, and $(\alpha, \delta)$ are of type $i$-1, $j$-1, and $k$-1, respectively. Furthemore, the analogous statements hold for the pairs $(\beta, \alpha), (\gamma, \alpha)$, $(\delta, \alpha) \in L(u) \times L(v)$ with $\beta, \gamma, \delta$ all distinct.
\end{lemma}
\begin{proof}
    Suppose for contradiction that $(\alpha, \beta)$, $(\alpha, 
\gamma)$, and $(\alpha, \delta)$ are of type 1-$i$, 1-$j$, and 1-$k$, respectively, for some $i,j,k \in [2]$. Let $L(x) = \{x_1, x_2, x_3\}$. Since $(\alpha, \beta)$ is type 1-$i$, we may assume without loss of generality that $\alpha \sim x_1$ and $\beta \sim x_2$. By the matching condition, it follows $\gamma \not\sim x_2$, and thus, as $(\alpha, \gamma)$ is of type 1-$j$, it follows that $\gamma \sim x_3$. Thus $\delta \sim x_1$, contradicting $(\alpha, \delta$ is of type 1-$k$. A similar proof shows it is not the case that $(\alpha, \beta)$, $(\alpha, 
\gamma)$, and $(\alpha, \delta)$ are of type $i$-1, $j$-1, and $k$-1, respectively, for some $i,j,k \in [2]$. A similar proof also shows the furthermore statement.
\end{proof}

We can now prove the lemma of this section.
\begin{lemma}
   $T$ is bad for at most 6 pairs in $L(u) \times L(v)$.
\end{lemma}
\begin{proof}
    Without loss of generality, assume $L(u) = \{1, 2, 3\}$ and $L(v) = \{4, 5, 6\}$. Suppose $T$ is bad for at least 7 pairs in $L(u) \times L(v)$; choose a subset $S$ of size 7 of these pairs, and let $S'$ be the remaining 2 pairs. Either the two pairs of $S'$ have a coordinate in common, or they do not. Therefore, by relabeling, we may without loss of generality assume that either $S' = \{(1, 4), (1, 5)\}$ or $S' = \{(1, 4), (2,5)\}$, and, in both cases, we have that $\{(1, 6), (2, 6), (3,4), (3, 5), (3, 6)\} \subseteq S$. 

    By Lemma \ref{lemma: not all 1's}, there does not exist $i, j, k \in \{1, 2\}$ such that $(1,6)$, $(2, 6)$, $(3, 6)$ are of type $i$-1, $j$-1, and $k$-1, respectively, and thus at least one of $\{(1,6), (2, 6), (3, 6)\}$ is of type 1-2. 
    Again by Lemma \ref{lemma: not all 1's}, it follows one of $\{(3,4), (3,5), (3,6)\}$ is of type 1-2.

    Suppose $(1,6)$ is of type 1-2. Then by Lemma \ref{lemma: not two twists}, it follows $(3,4)$ and $(3,5)$ are not of type 1-2. But as at least one $\{(3, 4), (3, 5), (3,6)\}$ is of type 1-2, it follows $(3,6)$ is of type 1-2. However, by Lemma \ref{lemma: not two 1-2's}, it cannot be that $(1,6)$ and $(3,6)$ are both of type 1-2, a contradiction. Thus $(1,6)$ is not of type 1-2. 
    A similar argument shows $(2,6)$ is not type 1-2. 
    And as at least one of $\{(1,6), (2,6), (3,6)\}$ is of type 1-2, it follows $(3,6)$ is type 1-2.

    By Lemma \ref{lemma: not all 1's}, since $(3,6)$ is of type 1-2, it follows at least one of $\{(3,4), (3,5)\}$ is of type 2-1. By Lemma \ref{lemma: not two twists}, it follows $(1,6)$ and $(2,6)$ are not of type 2-1, and thus are of type 1-$i$ and 1-$j$, respectively, for some $i, j \in \{1, 2\}$. But then $(1,6), (2,6), (3,6)$ are of types 1-$i$, 1-$j$, 1-2, respectively, contradicting Lemma \ref{lemma: not all 1's}. Therefore, it is not the case that $T$ is bad for at least 7 pairs in $L(u) \times L(v)$.
\end{proof}

\section{Proof of Lemma \ref{Lemma: good}}\label{Section: proof of good}

We explain below that it will suffice to prove the following:

\begin{lemma}\label{Lemma: good simple}
Let $\H$ be a correspondence cover for $T \setminus \{u, v\}$ with $\ell(w) = 2$ for all $w \in \{u_1, u_2, v_1, v_2\}$ and $1 \leq \ell(x), \ell(y), \ell(z) \leq 2$. If $\mathcal{H}$ is \textit{good}, then there exists a $1$-defective $\mathcal{H}$-coloring of $T \setminus \{u, v\}$.
\end{lemma}

  To see why this suffices to prove Lemma \ref{Lemma: good}, assume $\H_\phi$ in the statement of Lemma \ref{Lemma: good} is good. Then by Lemma \ref{Lemma: good simple}, there exists a $1$-defective $\H_\phi$-coloring of $T \setminus \{u,v\}$, say $\psi_0$. Take $\psi \defeq \psi_0 \cup \phi$. Since, by definition, $\im(\psi_0)$ does not include neighbors of $\phi(u)$ or $\phi(v)$, it follows $\psi$ is a $1$-defective $\H$-coloring of $T$ with $\defect_\psi(u)= \defect_\psi(v) = 0$.

  We therefore prove Lemma \ref{Lemma: good simple}. For convenience, in the rest of this section, let $T' \defeq T \setminus \{u, v\}$. Let $\H$ be a correspondence cover for $T'$ such that $\ell(w) = 2$ for all $w \in \{u_1, u_2, v_1, v_2\}$ and $1 \leq \ell(x), \ell(y), \ell(z) \leq 2$, and, furthermore, assume $\H$ is good. The proof of Lemma \ref{Lemma: good simple} then follow easily from the series of lemmas below, which cover all possible cases for the values of $\ell(x)$, $\ell(z)$, and $\ell(y)$.

\begin{lemma}\label{Lemma: 1-2-2}
     If $\ell(x) \geq 1$, $\ell(z) = 2$, and $\ell(y) = 2$, then there exists an $\H$ coloring of $T'$.
\end{lemma}
\begin{proof}
    Without loss of generality, we may assume $\ell(x) = 1$ by removing a color from $L(x)$ if necessary.
   By Lemma \ref{Lemma: 1-2}\ref{Lemma: 1-2 coloring}, there exists a $1$-defective $\H(R_1)$ coloring $\phi$ of $R_1$ with $\defect_{\phi}(z) = 0$. Let $\xi \in L(z) \setminus \{\phi(z)\}$.
Then by Lemma \ref{Lemma: 1-2}\ref{Lemma: 1-2 coloring}, there exists a $1$-defective $\mathcal{H}(R_2)\setminus\{\xi\}$ coloring $\psi$ of $R_2$. By definition, $\phi(z) = \psi(z)$, and as $\defect_\phi(z) = 0$, it follows $\phi \cup \psi$ is a $1$-defective $\H$-coloring of $T'$.
\end{proof}

\begin{lemma}\label{lemma: 2 - 2 - 1}
   If $\ell(x) = 2$, $\ell(z) = 2$, and $\ell(y)  \geq 1$, then there exists an $\H$ coloring of $T'$.
\end{lemma}
\begin{proof}
    This is the symmetric case to the statement of Lemma \ref{Lemma: 1-2-2}, and the proof follows similarly.
\end{proof}

\begin{lemma}\label{lemma: 2-1-2 not bad}
   If $\ell(x)  = 2$, $\ell(z) = 1$, and $\ell(y)  = 2$, then there exists an $\H$ coloring of $T'$.
\end{lemma}
\begin{proof}
    Since $\H$ is good, we may assume without loss of generality that $\H(R_1)$ is not wedged (the case when $\H(R_2)$ is not wedged is symmetric). Then by Lemma \ref{Lemma: 1-2}\ref{Lemma: 1-2 not wedged}, there exists a $1$-defective $\H(R_1)$ coloring $\phi$ of $R_1$ such that $\defect_{\phi}(z) = 0$. By Lemma \ref{Lemma: 1-2}\ref{Lemma: 1-2 coloring}, there exists a $1$-defective $\H(R_2)$ coloring $\psi$ of $R_2$.
    By definition, $\phi(z) = \psi(z)$, and as $\defect_\phi(z) = 0$, it follows $\phi \cup \psi$ is a $1$-defective $\H$-coloring of $T'$.
\end{proof}

\begin{lemma}\label{lemma: 1-1-2 not bad}
   If $\ell(x) = 1$, $\ell(z) =1$, and $\ell(y) = 2$, then there exists an $\H$ coloring of $T'$.
\end{lemma}
\begin{proof}
    Since $\H$ is good, it follows $\H(R_1)$ is not twisted and at least one of $\H(R_1)$ and $\H(R_2)$ is not wedged.

    Case 1: $\H(R_1)$ is not wedged. Then by Lemma \ref{Lemma: 1-1 not twisted or wedged}, there exists a $1$-defective $\H(R_1)$ coloring $\phi$ of $R_1$ with $\defect_\phi(z) = 0$. By Lemma \ref{Lemma: 1-2}\ref{Lemma: 1-2 coloring}, there exists a $1$-defective $\H(R_2)$ coloring $\psi$ of $R_2$. Since $\phi(z) = \psi(z)$ and $\defect_\phi(z) = 0$, it follows $\phi \cup \psi$ is a $1$-defective $H$-coloring of $T'$.
    
    Case 2: $\H(R_2)$ is not wedged. Then by Lemma \ref{Lemma: 1-2}\ref{Lemma: 1-2 not wedged}, there exists a $1$-defective $\H(R_2)$ coloring $\psi$ of $R_2$ with $\defect_{\psi}(z) = 0$. Since $\H(R_1)$ is not twisted, by Lemma \ref{Lemma: 1-1 not twisted or wedged}, there exists a $1$-defective $\H(R_1)$ coloring $\phi$ of $R_1$. Since $\phi(z) = \psi(z)$ and since $\defect_{\psi}(z) = 0$, it follows $\phi \cup \psi$ is a $1$-defective $\H$ coloring of $T'$.    
\end{proof}

\begin{lemma}\label{lemma: 2-1-1 not bad}
   If $\ell(x) = 2$, $\ell(z) =1$, and $\ell(y) = 1$, then there exists an $\H$ coloring of $T'$.
\end{lemma}
\begin{proof}
   This case is symmetric to that of Lemma \ref{lemma: 1-1-2 not bad}, and the proof follows similarly.
\end{proof}

\begin{lemma}\label{lemma: 1-2-1 not bad}
   If $\ell(x) = 1$, $\ell(z) =2$, and $\ell(y) = 1$, then there exists an $\H$ coloring of $T'$.
\end{lemma}
\begin{proof}
%
    Suppose $L(z) = \{z_1, z_2\}$. Case 1: there exist $i, j \in \{1, 2\}$ such that $\H(R_i) \setminus \{z_j\}$ is twisted. Without loss of generality, assume $i = 1$ and $j = 2$, so that  
    $\H(R_1) \setminus \{z_2\}$ is twisted. By Lemma \ref{lemma: not twisted and wedged}, there is a $0$-def $\H(R_1)$-coloring $\phi$ of $R_1$ with $\phi(z) = z_2$.
    Since $\H$ is good, it follows $\H(R_2)\setminus\{z_1\}$ is not twisted.
    Therefore, by Lemma \ref{Lemma: 1-1 not twisted or wedged}, there exists a $1$-defective $\H(R_2)\setminus\{z_1\}$ coloring $\psi$ of $R_2$. By definition, $\phi(z) = \psi(z)$, and as $\defect_\phi(z) = 0$, it follows $\phi \cup \psi$ is a $1$-defective $\H$-coloring of $T'$.

    Case 2: $\H(R_i) \setminus \{z_j\}$ is not twisted for all $i,j \in \{1, 2\}$. By Lemma \ref{Lemma: 1-2}\ref{Lemma: 1-2 coloring}, there exists an $\H(R_1)$ coloring $\phi$ of $R_1$ such that $\defect_\phi(z) = 0$. Without loss of generality, let $\phi(z) = z_1$. Then as $\H(R_2)\setminus\{z_2\}$ is not twisted, by Lemma \ref{Lemma: 1-1 not twisted or wedged} there exists a $1$-defective $\H(R_2)\setminus\{z_2\}$ coloring $\psi$ of $R_2$. By definition, $\phi(z) = \psi(z)$, and as $\defect_\phi(z) = 0$, it follows $\phi \cup \psi$ is a $1$-defective $\H$ coloring of $T'$.
\end{proof}

\begin{lemma}\label{lemma: 1-1-1 not bad}
   If $\ell(x) = 1$, $\ell(z) =1$, and $\ell(y) = 1$, then there exists an $\H$ coloring of $T'$.
\end{lemma}
\begin{proof}
    Since $\H$ is good, it follows $\H(R_1)$ is not twisted, $\H(R_2)$ is not twisted, and at least one of $\H(R_1)$ and $\H(R_2)$ is not wedged. Without loss of generality, we may assume $\H(R_1)$ is not wedged (the case when $\H(R_2)$ is not wedged is symmetric). Then by Lemma \ref{Lemma: 1-1 not twisted or wedged}, there exists a $1$-defective $\H(R_1)$ coloring $\phi$ of $R_1$ with $\defect_\phi(z) = 0$. Since $\H(R_2)$ is not twisted, by Lemma \ref{Lemma: 1-1 not twisted or wedged}, there exists a $1$-defective $\H(R_2)$ coloring $\psi$ of $R_2$. Since $\phi(z) = \psi(z)$ and $\defect_\phi(z) = 0$, it follows $\phi \cup \psi$ is a $1$-defective $H$-coloring of $T'$.  
\end{proof}

\section{Proof of Lemma \ref{lemma: bad}}\label{section: proof of lemma bad}
 
    The proof of Lemma \ref{lemma: bad} follows from the series of lemmas below, which cover all the cases when $\H_\phi$ is bad. For convenience, throughout this section assume $\H$ is such that $L(x) = \{x_1, x_2, x_3\}$, $L(y) = \{y_1, y_2, y_3\}$, $L(z) = \{z_1, z_2, z_3\}$, $L(u_1) = \{b_1, b_2, b_3\}$,  $L(v_1) = \{d_1, d_2, d_3\}$, $L(u_2) = \{e_1, e_2, e_3\}$,  $L(v_2) = \{f_1, f_2, f_3\}$. Further assume $\phi(u) = \alpha$, $\phi(v) = \beta$ for some $\alpha \in L(u)$ and $\beta \in L(v)$.

\begin{lemma}\label{lemma: 2-1-2 wedged}
If $\H_\phi$ is 1-bad, (i.e., $\ell_\phi(x) = 2$, $\ell_\phi(z) = 1$, $\ell_\phi(y) = 2$, and $\H_\phi(R_1)$ and $\H_\phi(R_2)$ are both wedged), then $\phi$ can be extended to $1$-defective $\H$-colorings $\psi, \rho$ of $T$ such that $\defect_\psi(u) \leq 1$, $\defect_\psi(v) = 0$, and $\defect_\rho(u) = 0$, $\defect_\rho(v) \leq 1$.
\end{lemma}
\begin{proof}
We first prove the existence of $\psi$. Without loss of generality, assume $\{x_3, b_3, z_1\} \subseteq N_H(\alpha)$ and $\{x_3, d_3, z_3\} \subseteq N_H(\beta)$; furthermore, as $\H_\phi(R_1)$ is wedged, we may assume $z_2 \sim b_2$, $z_2 \sim d_2$, $x_1 \sim b_1$, $b_1 \sim d_1$, and $d_1 \sim x_2$ (see Figure \ref{Fig: 2-1-2_Bad}).

\begin{figure}
    \centering

\includegraphics{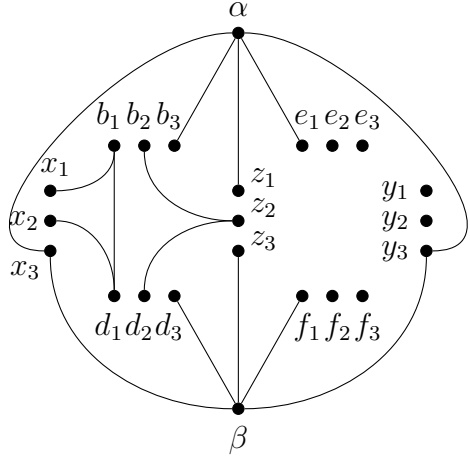}
    \caption{1-Bad cover $\mathcal{H}$ as in Lemma \ref{lemma: 2-1-2 wedged} (some edges and vertices not shown).}
    \label{Fig: 2-1-2_Bad}

\end{figure}

By Lemma \ref{Lemma: 1-2}\ref{Lemma: 1-2 coloring}, there exists a $1$-defective $\H_\phi(R_2)$-coloring $\psi_0$ of $R_2$. Note that $\psi_0(z) = z_2$. Now extend $\psi_0$ by taking $\psi_0(x) = x_1$, $\psi_0(u_1) = b_3$, $\psi_0(v_1) = d_1$. Since $\{x_1, b_3, d_1, z_2\}$ is an independent set, it follows $\defect_{\psi_0}(z_2) \leq 1$, and $\defect_{\psi_0}(u_1) = 0$. Take $\psi \defeq \phi \cup \psi_0$. Then $\psi$ is a $1$-def $\H$-coloring with $\psi(u) \leq 1$ and $\psi(v) = 0$, as desired.
The coloring $\rho$ can be constructed similarly, switching the roles of $u$ with $v$ and $b_3$ with $d_3$.
\end{proof}

\begin{lemma}\label{lemma: 1-1-2 and R_1 is twisted}
If $\H$ is 2(i)-bad (i.e., $\ell_\phi(x) = 1$, $\ell_\phi(z) = 1$, $\ell_\phi(y) = 2$, and $\H_\phi(R_1)$ is twisted), then $\phi$ can be extended to $1$-defective $\H$-colorings $\psi, \rho$ of $T$ such that $\defect_\psi(u) \leq 1$, $\defect_\psi(v) = 0$, and $\defect_\rho(u) = 0$, $\defect_\rho(v) \leq 1$.
\end{lemma}
\begin{proof}
We will prove the existence of $\psi$ and the existence of $\rho$ will follow similarly.
Without loss of generality, assume $\{x_1, b_3, z_1\} \subseteq N_H(\alpha)$ and $\{x_3, d_3, z_3\} \subseteq N_H(\beta)$; furthermore, as $\H_\phi(R_1)$ is twisted, we may assume $x_2b_1d_2z_2b_2d_1x_2$ is a cycle in $H$ (see Figure \ref{Fig: 1-1-2_Bad_Twisted}).

By Lemma \ref{Lemma: 1-2}\ref{Lemma: 1-2 coloring}, there exists a $1$-defective $\H_\phi(R_2)$-coloring $\psi_0$ of $R_2$. Note that $\psi_0(z) = z_2$. Extend $\psi_0$ by setting $\psi_0(x) = x_2$, $\psi_0(u_1) = b_3$, $\psi(v_1) = d_1$. Then, due to the matching condition, $\psi_0$ is a $1$-def $\H$-coloring of $T\setminus\{u, v\}$, with $\defect_{\psi_0}(u_1) = 0$. Then $\psi \defeq \phi \cup \psi_0$ is a $1$-def $\H$-coloring of $T$ with $\defect_\psi(u) = 1$, and $\defect_\psi(v) = 0$. The coloring $\rho$ can be constructed similarly, switching the roles of $u$ with $v$ and $b_3$ with $d_3$.

\end{proof}

\begin{figure}
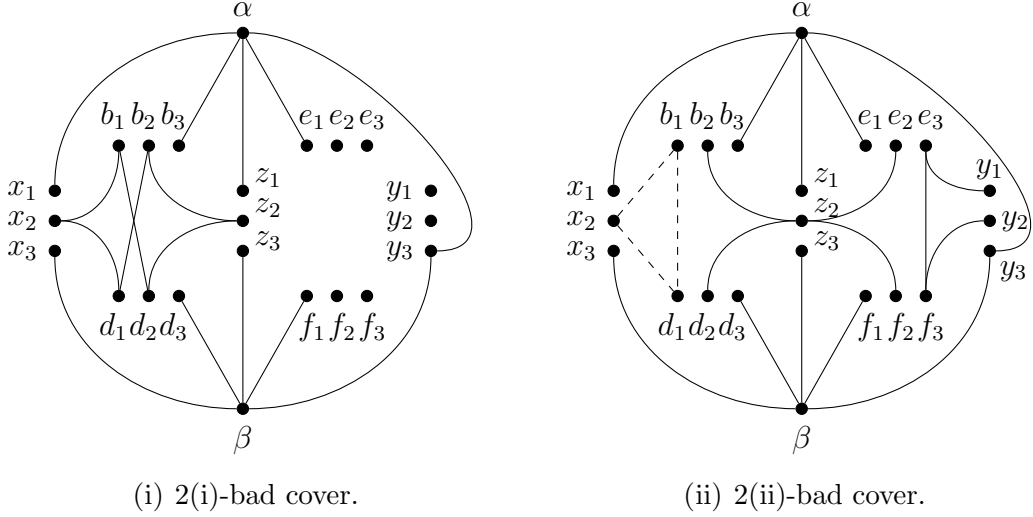

    \centering
    \begin{subfigure}[t]{.45\textwidth}
  \includegraphics{figures/Figure_2BadCase1.tex}
    \caption{2(i)-bad cover.}
    \label{Fig: 1-1-2_Bad_Twisted}
    \end{subfigure}%
    \begin{subfigure}[t]{.45\textwidth}
        \centering
    \includegraphics{figures/Figure_2BadCase2.tex}
    \caption{2(ii)-bad cover.}
    \label{Fig: 1-1-2_Wedged}
    \end{subfigure}
\caption{2-bad covers as in Lemmas \ref{lemma: 1-1-2 and R_1 is twisted} and \ref{lemma: 1-1-2 and R_1 and R_2 are wedged} respectively (some edges and vertices not shown).}
\end{figure}

\begin{lemma}\label{lemma: 1-1-2 and R_1 and R_2 are wedged}
If $\H$ is 2(ii)-bad (i.e., $\ell_\phi(x) = 1$, $\ell_\phi(z) = 1$, $\ell_\phi(y) = 2$, and $\H_\phi(R_1)$ and $\H_\phi(R_2)$ are wedged), then $\phi$ can be extended to $1$-defective $\H$-colorings $\psi, \rho$ of $T$ such that $\defect_\psi(u) \leq 1$, $\defect_\psi(v) = 0$, and $\defect_\rho(u) = 0$, $\defect_\rho(v) \leq 1$.
\end{lemma}
\begin{proof}
We will prove the existence of $\psi$ and the existence of $\rho$ will follow similarly. If $\H_\phi(R_1)$ is twisted, then we may apply Lemma \ref{lemma: 1-1-2 and R_1 is twisted}; thus, we may assume $\H_\phi(R_1)$ is not twisted.
Without loss of generality, assume $\{y_3, e_1, z_1\} \subseteq N_H(\alpha)$ and $\{y_3, f_1, z_3\} \subseteq N_H(\beta)$; furthermore, as $\H_\phi(R_2)$ is wedged, we may assume without loss of generality that $z_2 \sim e_2$, $z_2 \sim f_2$, and $y_1e_3f_3y_2$ is a path in $H$ (see Figure \ref{Fig: 1-1-2_Wedged}).

Since $\H_\phi(R_1)$ is not twisted, by Lemma \ref{Lemma: 1-1 not twisted or wedged}, there exists a $1$-defective $\H_\phi(R_1)$-coloring $\psi_0$ of $R_1$ with $\defect_{\psi_0}(z) = 0$. Note $\psi_0(z) = z_2$.
Then due to the matching condition, $\{y_1, e_1, z_2, f_3\}$ is an independent set in $H$; extend $\psi_0$ by selecting this independent set. Then $\defect_{\phi_0}(u_2) = \defect_{\phi_0}(z) = 0$.  Thus $\psi \defeq \psi_0 \cup \phi$ is a $1$-defective $\H$ coloring of $T$ with $\defect_\psi(u) = 1$ and $\defect_\phi(v) = 0$, as desired. The coloring $\rho$ can be constructed similarly by swapping the roles of $u_2$ with $v_2$ and $e_1$ with $f_1$.


\end{proof}

\begin{lemma}\label{lemma: 2-1-1 twisted}
If $\H$ is 3(i)-bad (i.e., $\ell_\phi(x) = 2$, $\ell_\phi(z) = 1$, $\ell_\phi(y) = 1$, and $\H_\phi(R_2)$ is twisted), then $\phi$ can be extended to $1$-defective $\H$-colorings $\psi, \rho$ of $T$ such that $\defect_\psi(u) \leq 1$, $\defect_\psi(v) = 0$, and $\defect_\rho(u) = 0$, $\defect_\rho(v) \leq 1$.
\end{lemma}
\begin{proof}
    This is symmetric to Lemma \ref{lemma: 1-1-2 and R_1 is twisted}.
\end{proof}

\begin{lemma}\label{lemma: 2-1-1 wedged}
If $\H$ is 3(ii)-bad (i.e., $\ell_\phi(x) = 2$, $\ell_\phi(z) = 1$, $\ell_\phi(y) = 1$, and $\H_\phi(R_1)$ and $\H_\phi(R_2)$ are wedged), then $\phi$ can be extended to $1$-defective $\H$-colorings $\psi, \rho$ of $T$ such that $\defect_\psi(u) \leq 1$, $\defect_\psi(v) = 0$, and $\defect_\rho(u) = 0$, $\defect_\rho(v) \leq 1$.
\end{lemma}
\begin{proof}
    This is symmetric to Lemma \ref{lemma: 1-1-2 and R_1 and R_2 are wedged}.
\end{proof}

\begin{lemma}\label{lemma: 1-2-1 bad}
If $\H$ is 4-bad (i.e., $\ell_\phi(x) = 1$, $\ell_\phi(z) = 2$, $\ell_\phi(y) = 1$, and there exists distinct $z_i, z_j \in L(z)$ such that $\H(R_1) \setminus \{z_i\}$ and $\H(R_2) \setminus \{z_j\}$ are twisted), then $\phi$ can be extended to $1$-defective $\H$-colorings $\psi, \rho$ of $T$ such that $\defect_\psi(u) \leq 1$, $\defect_\psi(v) = 0$, and $\defect_\rho(u) = 0$, $\defect_\rho(v) \leq 1$.
\end{lemma}
\begin{proof}
    Without loss of generality, we may assume
    \begin{align*}
        \{x_1, b_3, z_3, e_1, y_1\} &= N_H(\alpha),\\
        \{x_3, d_3, z_3, f_1, y_3\} &= N_H(\beta).
    \end{align*}
    Furthermore, we may assume that $i = 2$ and $j = 1$. Since $\H_\phi(R_1)\setminus\{z_2\}$ is twisted, and $\H_\phi(R_2)\setminus\{z_1\}$ is twisted, we may assume without loss of generality that $x_2b_1d_2z_1b_2d_1x_2$ and $y_2e_3f_2z_2e_2f_3y_2$ are cycles in $H$ (see Figure \ref{fig: 1-2-1 Bad}).
    \begin{figure}
        \centering
        \includegraphics{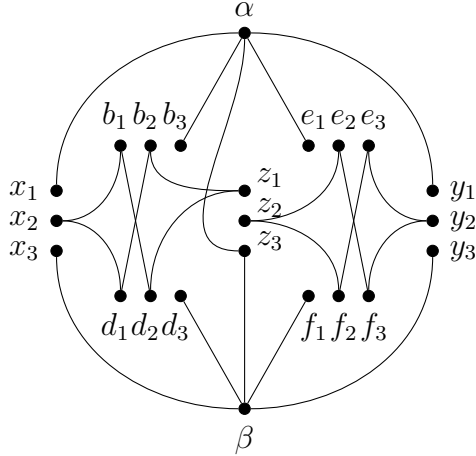}
        \caption{4-bad cover $\mathcal{H}$ as in Lemma \ref{lemma: 1-2-1 bad} (some edges and vertices not shown).}
        \label{fig: 1-2-1 Bad}
    \end{figure}
Then, due to the matching condition, it follows if $\psi$ is the $\H$-coloring defined by $\{z_1, b_3, d_1, x_2, e_2, f_2, y_2, \alpha, \beta\}$, then $\psi$ is $1$-defective with $\psi(u) =1$ and $\psi(v) = 0$. The coloring $\rho$ can be constructed similarly by swapping the roles of $b_3, d_1$ with $d_3, b_1$.
\end{proof}

\begin{lemma}\label{lemma: 1-1-1 bad}
If $\H$ is 5-bad (i.e., $\ell_\phi(x) = 1$, $\ell_\phi(z) = 1$, $\ell_\phi(y) = 1$, and one of the following holds: $\H_\phi(R_1)$ is twisted, $\H_\phi(R_2)$ is twisted, or both $\H_\phi(R_1)$ and $\H_\phi(R_2)$ are wedged), then $\phi$ can be extended to $1$-defective $\H$-colorings $\psi, \rho$ of $T$ such that $\defect_\psi(u) \leq 1$, $\defect_\psi(v) = 0$, and $\defect_\rho(u) = 0$, $\defect_\rho(v) \leq 1$.
\end{lemma}
\begin{proof}
 
Without loss of generality, assume $\{x_1, b_3, z_1, e_1, y_1\} = N_H(\alpha)$ and $\{x_3, d_3, z_3, f_1, y_3\} = N_H(\beta)$. 

Assume $\H_\phi(R_1)$ is twisted. Then without loss of generality, we may assume $x_2b_1d_2z_2b_2d_1x_2$ is a cycle in $H$. If $\H_\phi(R_2)$ is twisted, we may assume $y_2e_3f_2z_2e_2f_3y_2$ is a cycle in $H$ (see Figure \ref{Fig: 1-1-1_Bad}).
\begin{figure}
    \centering
    \begin{subfigure}[t]{.5\textwidth}
    \includegraphics{figures/Figure_5BadCase1.tex}
    \caption{}
    \label{Fig: 1-1-1_Bad}
    \end{subfigure}%
    \begin{subfigure}[t]{.5\textwidth}
        \centering
    \includegraphics{figures/Figure_5BadCase2.tex}
    \caption{}
    \label{fig: 1-1-1_Bad Wedged}
    \end{subfigure}
    \caption{5-bad covers $\H$ as in Lemma \ref{lemma: 1-1-1 bad} (some edges and vertices not shown).}
\end{figure}
Then due to the matching condition, if $\psi$ is the coloring induced by $\{x_2, b_2, d_2, z_1, e_2, f_2, y_2, \alpha, \beta\}$, then $\psi$ is a $1$-defective $\H$ coloring $\phi$ of $T$ with $\defect_\phi(u)= 1$ and $\defect_\phi(v) = 0$. The coloring $\rho$ can be constructed in a similar manner (take $z_3$ instead of $z_1$). So we may assume $\H_{\rho}(R_2)$ is not twisted, and thus by Lemma \ref{Lemma: 1-1 not twisted or wedged}, there is a $1$-defective $\H_\phi(R_2)$ coloring, $\psi_0$ of $R_2$. Note $\psi_0(z) = z_2$. Extend $\psi_0$ by taking $\{b_1, d_1, x_1\}$. Then $\psi \defeq \psi_0 \cup \phi$ is a $1$-defective $\H$ coloring of $T$ with $\defect_\psi(u) = 1$ and $\defect_\psi(v) = 0$. The coloring $\rho$ can be constructed similarly, taking $x_3$ in place of $x_1$. Thus we may assume $\H_\phi(R_1)$ is not twisted. By symmetry, we may assume $\H_\phi(R_2)$ is not twisted.

Thus we may assume $\H_\phi(R_1)$ and $\H_\phi(R_2)$ are not twisted, and that both $\H_\phi(R_1)$ and $\H_\phi(R_2)$ are wedged. Then, without loss of generality, we may assume $z_2 \sim b_2$, $z_2 \sim d_2$, and at least two of $\{x_2b_1, b_1d_1, d_1x_2\}$ are edges in $H$. Similalry, we may assume $z_2 \sim e_2$, $z_2 \sim f_2$, and at least two of $\{y_2e_3, e_3f_3, f_3y_2\}$ are edges in $H$ (see Figure \ref{fig: 1-1-1_Bad Wedged}).

Since $\H_\phi(R_2)$ is not twisted, by Lemma \ref{Lemma: 1-1 not twisted or wedged} there exists a $1$-defective $\H_\phi(R_2)$ coloring $\psi_0$ of $R_2$.
Since at least two of $\{x_2b_1, b_1d_1, d_1x_2\}$ are edges in $H$, there are three cases to consider.

Case 1: $x_2 \sim b_1$ and $x_2 \sim d_1$. Then, due to the matching condition, $x_1 \notsim b_1$ and $x_1 \notsim d_1$. 
Thus if $\psi$ is the coloring induced by $\im(\psi_0) \cup \{x_1, b_1, d_1\} \cup \{\alpha, \beta\}$, then $\psi$ is a $1$-defective $\H$ coloring of $T$ with $\defect_\psi(u) = 1$ and $\defect_\psi(v) = 0$. 

Case 2: $x_2 \sim b_1$ and $b_1 \sim d_1$. If $x_2 \sim d_1$, then we are in case 1, so we may assume $x_2 \notsim d_1$. Note by the matching condition that $x_2 \notsim b_3$, and $d_1 \notsim b_3$. Thus if $\psi$ is the coloring induced by  $\im(\psi_0) \cup \{x_2, d_1, b_3\} \cup \{\alpha, \beta\}$, then $\phi$ is a $1$-defective $\H$ coloring of $T$ with $\defect_\psi(u) = 1$ and $\defect_\psi(v) = 0$. 


Case 3: $x_2 \sim d_1$ and $d_1 \sim b_1$.
Note that by the matching condition, $b_3 \notsim d_1$ and $b_3 \notsim z_2$. If $b_3 \notsim x_2$, then take $\psi$ to be the coloring induced by $\psi_0 \cup \{x_2, b_3, d_1\} \cup \{\alpha, \beta\}$, and observe $\psi$ is as desired. Thus we may assume $b_3 \sim x_2$. If $x_1 \notsim b_1$, then we can take $\psi$ to be $\psi_0 \cup \{x_1, b_1, d_1\} \cup \{\alpha, \beta\}$, and $\psi$ is as desired. So we may assume $x_1 \sim b_1$.

At this point, we shift our focus to $R_2$ and apply the same arguments as above to $R_2$. That is, since $\H_\phi(R_2)$ is not twisted, by Lemma \ref{Lemma: 1-1 not twisted or wedged}, there exists a $1$-defective $\H_\phi(R_2)$ coloring $\psi_0$ of $R_2$.
Since at least two of $\{y_2e_3, e_3f_3, f_3y_2\}$ are edges in $H$, there are three cases to consider. These three cases are similar to those above; we may therefore assume we end in the third case with $y_2\sim f_3$, $e_3 \sim f_3$, $y_1 \sim e_3$, $y_2\sim e_1$. Therefore, $H$ is as in Fig. \ref{fig: 1-1-1 wedged super bad}.
\begin{figure}
    \centering
    \includegraphics[scale = 1]{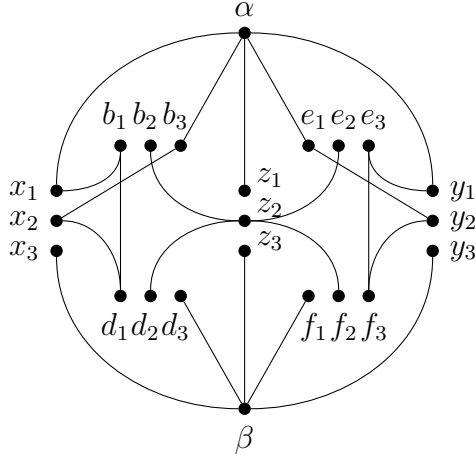}
    \caption{$H$ as in the last case of the proof of Lemma \ref{lemma: 1-1-1 bad} (some edges and vertices not shown).}
    \label{fig: 1-1-1 wedged super bad}
\end{figure}
Taking $\psi$ to be the coloring induced by $\{z_1, b_2, d_2, x_2, e_2, f_2, y_2, \alpha, \beta\}$ results in a $1$-defective $\H$-coloring of $T$ with $\defect_\psi(u) = 1$ and $\defect_\psi(v) = 0$. Constructing $\rho$ is similar.

\end{proof}

\section{Proof of Theorem \ref{Theorem: T(8) not $4$-correspondable}}\label{Section: T(8) not $4$-correspondable}
We construct a 4-fold correspondence cover $\H$ of $T(4)$ such that $T(4)$ is not $\H$-colorable. 

\begin{figure}
    \centering
    \includegraphics{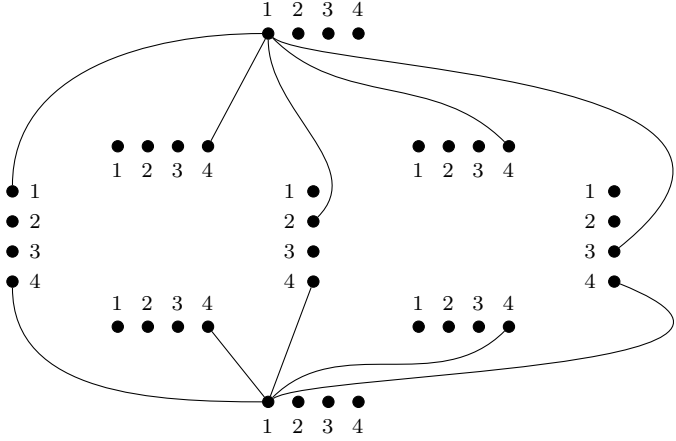}
    \caption{$H_{1, 2, 3;4}(1,1)$ (edges in identity matchings not shown).}
    \label{fig: T not H colorable}
\end{figure}
We start by constructing covers for $T$. Let $L \defeq V(T) \times [4]$.
For variables $\alpha,\beta \in [4]$, let $H_{1,2,3;4}(\alpha,\beta)$ be the cover graph for $T$ defined as follows (see Fig. \ref{fig: T not H colorable}):
\begin{enumerate}
    \item $V(H_{1,2,3;4}(\alpha,\beta)) = V(T) \times [4]$;
    \item $\forall\, c \in \{(x,1), (z, 2), (y, 3), (u_1, 4), (v_1, 4)\}$ we have $(u, \alpha) \sim c$;
    \item $\forall\, w \in N_G(v) \big((v,\beta) \sim (w, 4)\big)$.
    \item All edges in $T\setminus\{u,v\}$ have identity matchings.
\end{enumerate}
We can similarly define $H_{i, j, k; \ell}(\alpha, \beta))$ for any distinct $i, j, k, \ell \in [4]$ and $\alpha,\beta \in [4]$ using the same process as above. Then $(L, H_{i,j,k;\ell}(\alpha, \beta))$ is a correspondence cover for $T$.

\begin{lemma}\label{lemma: T minus alpha and beta is not list colorable}
For any $\alpha, \beta \in [4]$, and any distinct $i, j, k, \ell \in [4]$, there does not exist an $(L, H_{i,j,k;\ell}(\alpha, \beta))$-coloring $\phi$ of $T$ with $\phi(u) = (u,\alpha)$ and $\phi(v) = (v,\beta)$.
\end{lemma}
\begin{proof}
   For concreteness, say $(i, j, k, \ell) = (1, 2, 3, 4)$.
    Suppose $T$ were $(L, H_{1,2,3;4}(\alpha, \beta))$-colorable, say by $\phi$, with $\phi(u) = (u, \alpha)$ and $\phi(v) = (v,\beta)$. If $\phi(z) = (z,1)$, then it follows $\phi$ is a proper list coloring of the triangle $u_1, v_1, x$ with $L(w) = \{(w, 2), (w, 3)\}$ for each $w \in \{u_1, v_1, x\}$, a clear contradiction. Similarly, if $\phi(z) = 3$, then $u_2, v_2, y$ cannot be colored. Thus, no such $\phi$ exists. 
\end{proof}

Let $\sigma$ be a permutation of $[4]$. We now define a correspondence cover $\mathcal{H}(\sigma) \defeq (L, H(\sigma))$ for $T$ as follows: let $L(w) = \{w\} \times [4]$ for all $w \in V(T)$ and let 
\[ H(\sigma) \defeq H_{1,2,3;4}(1, \sigma(1)) \cup H_{2,3,4;1}(2, \sigma(2)) \cup H_{3,4,1;2}(3, \sigma(3)) \cup H_{4,1,2;3}(4, \sigma(4)).\]

It is clear that $H(\sigma)$ is indeed a correspondence cover; indeed, if not, then it is an edge with endpoint $u$ or endpoint $v$ that violates the matching condition. But this is not possible.

\begin{lemma}\label{lemma: bad for 4 pairs}
   For all $\sigma \in S_4$, there does not exist an $\H(\sigma)$ coloring $\phi$ of $T$ with $\phi(u) = (u,i)$ and $\phi(v) = (v,\sigma(i))$.
\end{lemma}

\begin{proof}
    Suppose not, and let $\sigma \in S_4$ and $\phi$ be a $\mathcal{H}(\sigma)$ coloring such that $\phi(u) = (u, i)$ and $\phi(v) = (v, \sigma(i))$ for some $i \in [4]$. Then $\phi$ is an $H_{i, i+1, i+2; i+3}(i, \sigma(i))$ coloring (where addition here is the operation $i + j \defeq (i-1 + j \mod 4) + 1$). By Lemma \ref{lemma: T minus alpha and beta is not list colorable}, such a $\phi$ cannot exist.
\end{proof}

\begin{theorem}
    $T(4)$ is not $4$-correspondable.
\end{theorem}
\begin{proof}
    We construct a 4-fold cover $\H = (L, H)$ such that $T(4)$ is not $\H$-colorable. 
    Let $L(w)  = \{w\} \times [4]$ for every $w \in V(T(4))$. 
    
   Let $\{A_1, A_2, A_3, A_4\}$ be a partition of the 16 elements of $[4]\times [4]$ such that each list is a permutation of $[4]$. For instance:
    \[\begin{tabular}{c|c|c|c}
        $A_1$ & $A_2$ & $A_3$ & $A_4$\\
        \hline
        $(1, 1)$ & $(1, 4)$ & $(1, 3)$ & $(1, 2)$\\
        $(2, 2)$ & $(2, 1)$ & $(2, 4)$ & $(2, 3)$\\
        $(3, 3)$ & $(3, 2)$ & $(3, 1)$ & $(3, 4)$\\
        $(4, 4)$ & $(4, 3)$ & $(4, 2)$ & $(4, 1)$    
    \end{tabular}
    \]
    Now define a correspondence cover $(L, H)$ for $T(4)$ by taking the cover for each $T_i$ to be $H(A_i)$, $i \in [4]$. Suppose $\phi$ is a $(L,H)$-coloring of $T(4)$. Suppose $\phi(u) = (u, i)$ and $\phi(v) = (v, j)$ for some $i,j \in [4]$. Let $k \in [4]$ be the index such that $(i, j) \in A_k$. Since $\phi$ is an $(L,H)$ coloring, it follows $\phi$ is a $H(A_k)$-coloring of $T_k$, contradicting Lemma \ref{lemma: bad for 4 pairs}. Thus $T(4)$ is not $(L, H)$-colorable. Since $(L,H)$ is 4-fold, it follows $T(4)$ is not $4$-correspondable.
    \end{proof}

\section{Future directions and open questions}\label{section: open}
We discuss some open questions. The first natural question is to finish characterizing the pairs $(d,k)$ such that all planar graphs are $d$-def $k$-correspondable by answering the following:
\begin{question}
     What is the minimum $d \in \N$ such that all planar graphs are $d$-def $3$-correspondable? (It is known $4 \leq d \leq 6$.)
\end{question}
\begin{question}
    What is the minimum $d \in \N$ such that all planar graphs are $d$-def $4$-correspondable? (It is known $1 \leq d \leq 2$.)
\end{question}

Another natural question is the relationship among defective coloring, defective list coloring, and defective correspondence coloring. For instance, \v{S}rekovski \cite[Problem 3.4]{vskrekovski1999list} asked if every $3$-colorable planar graph is $1$-def $3$-choosable; we ask the stronger question below:
\begin{question}
    Are all $3$-colorable planar graphs $1$-def $3$-correspondable?
\end{question}
Along these lines, we can also ask the following strengthening of Theorem \ref{theorem: 1d3corr but not 4corr}:
\begin{question}
    Does there exist a planar graph that is $1$-def $3$-correspondable, but not 4-\textit{choosable}?
\end{question}

It is also natural to ask how these results extend to various compact surfaces. For instance, Cowen--Cowen--Woodall \cite{cowen1986defective} proved that every graph embeddable in compact surface $S$ with Euler characteristic $\chi(S)$ is $p(\chi(S))$-def $4$-colorable, where $p$ is a linear polynomial. Woodall extended this result to list coloring in \cite{woodall2011defective}. Thus we ask the following:

\begin{question}
   Does there exist a linear polynomial $p$ such that every graph embeddable in a compact surface $S$ is $p(\chi(S))$-def $4$-correspondable?
\end{question}

\section*{Acknowledgments}
Thank you to Anton Bernshteyn for helpful comments improving this manuscript.

\printbibliography

\end{document}